\documentclass{article}

\usepackage{amsmath, amssymb, amsthm}
\usepackage{algorithm, algpseudocode}
\usepackage{hyperref, cleveref}
\usepackage{subcaption}
\usepackage{geometry}
\usepackage{graphicx}
\usepackage{listings}
\usepackage{authblk}
\usepackage{pgffor}
\usepackage{natbib}
\title{Reconstruction of Space-Dependent Thermal Conductivity from Sparse Temperature Measurements}

\author[1]{Guangting Yu}
\author[1]{Shiwei Lan}
\author[2]{Kookjin Lee}
\author[1]{Alex Mahalov}

\newcommand{\email}[1]{\href{mailto:#1}{#1}}

\affil[1]{School of Mathematical and Statistical Sciences, Arizona State University, Tempe, AZ 85281 USA
\email{guangtin@asu.edu}, \email{slan@asu.edu}, \email{mahalov@asu.edu}}

\affil[2]{School of Computing and Augmented Intelligence, Arizona State University, Tempe, AZ 85281 USA \email{kookjin.lee@asu.edu}}

\newenvironment{keywords}{
\noindent\textbf{Keywords:}
}{\vspace{1ex}}

\newenvironment{MSCcodes}{
\noindent\textbf{Mathematics Subject Classification (2020):}
}{\vspace{1ex}}

\newtheorem{theorem}{Theorem}[section]
\newtheorem{proposition}[theorem]{Proposition}
\newtheorem{lemma}[theorem]{Lemma}

\theoremstyle{remark}
\newtheorem*{remark}{Remark}

\begin{document}

\maketitle

% REQUIRED
\begin{abstract}
We present a novel method for reconstructing thermal conductivity coefficients in 1D and 2D heat equations using moving sensors that dynamically traverse a domain to record sparse and noisy temperature measurements.
We significantly reduce the computational cost associated with forward  evaluations of partial differential equations by employing automatic differentiation, enabling a more efficient and scalable reconstruction process.
This allows inverse problems to be solved with fewer sensors and observations.
Specifically, we demonstrate successful reconstruction of thermal conductivity on the 1D circle and 2D torus, using one and four moving sensors, respectively, with their positions recorded over time.
Our method incorporates sampling algorithms to compute confidence intervals for the reconstructed conductivity, improving robustness against measurement noise.
Extensive numerical simulations of heat dynamics validate the efficacy of our approach, confirming both the accuracy and stability of the reconstructed thermal conductivity.
Additionally, the method is thoroughly tested using large datasets from machine learning, allowing us to evaluate its performance across various scenarios and ensure its reliability.
This approach provides a cost-effective and flexible solution for conductivity reconstruction from sparse measurements, making it a robust tool for solving inverse problems in complex domains.
\end{abstract}

% REQUIRED

% REQUIRED
\section{Introduction}\label{sec1}

This paper addresses the numerical solution of the inverse heat equation, specifically focusing on reconstructing thermal conductivity from temperature sensor measurements.
Thermal conductivity, a critical property governing heat diffusion, is essential in various industrial applications, including cooling of microchips \citep{Li2022}, temperature regulation in chemical reactors, and the optimization of heat-dissipating devices \citep{CAO2018150}.
These applications require precise temperature control to ensure the efficiency, safety, and longevity of the components involved.

While thermal conductivity values for homogeneous materials like metals or liquids are well-documented \citep{YNagasaka_1981}, measuring conductivity in inhomogeneous materials with spatially varying properties poses a significant challenge.
Traditional methods often require physically breaking down the material into smaller, homogeneous pieces, which can compromise the material's integrity and yield inaccurate measurements.

Recent advancements in sensor technology have enabled the development of non-invasive methods to overcome these limitations.
\citet{Stuart_2010} provides a method within a Bayesian framework and applies it to reconstruct the thermal conductivity of a one-dimensional rod.
However, this approach requires many sensors on the rod to achieve accurate reconstruction.
\citet{BESKOS2017327} extends the Bayesian framework to two-dimensional non-convex domains, but the requirement for numerous sensors remains.
\citet{Zhang2013fractional} proposed the homotopy regularization algorithm with sigmoid-type homotopy parameter to simultaneously invert the space-dependent diffusion coefficient and source magnitude, but it works only for one-dimensional spatial domain.
More recently, \citet{Boumenir2023} proposed a promising alternative method using multiple temperature measurements with only a single sensor to accurately reconstruct the thermal conductivity of a one-dimensional rod.
However, it relies on multiple physical experiments with different initial conditions, which can be costly and time-consuming.

To address these challenges, we propose a novel method utilizing moving sensors that dynamically traverse the domain while recording both temperature and position data.
In one-dimension domains, this simply means the sensor is sliding on the rod.
This approach fully utilizes the sensor's ability to record temperature over time, forming a time series of temperature and position measurements.
This enables the effective reconstruction of the thermal conductivity from a single physical experiment, completely eliminating the need for multiple experiments and significantly reduces the number of sensors.
By leveraging machine learning techniques and automatic differentiation, we optimize the reconstruction process and minimize the computational costs associated with forward evaluations of partial differential equations (PDEs).
These advancements make our method efficient, scalable, and practical for real-world applications.
Furthermore, we extend this method to two-dimensional domains, such as the two-dimensional torus, where fewer moving sensors are required to achieve accurate reconstructions, thus lowering overall costs.
A practical application could involve reconstructing the thermal conductivity of a city using temperature sensors mounted on vehicles, with GPS tracking providing accurate positional data.

We rigorously validate our method through numerical simulations and large-scale testing on machine learning datasets, ensuring its robustness across various scenarios.
Our theoretical contributions include a thorough analysis of numerical stability, leveraging the spectral properties of the underlying problem to ensure robustness against errors in the conductivity reconstruction.
Additionally, this work builds on and extends the research previously published by \citet{Boumenir2023} and \citet{https://doi.org/10.1002/mma.7601} further advancing techniques in thermal conductivity reconstruction for inhomogeneous materials.
Given the significance of this method, we believe it offers a cost-effective and practical solution for a wide range of applications, demonstrating its potential for broader use in real-world scenarios.

% The outline is not required, but we show an example here.
% {\color{red}
% The paper is organized as follows.
% }

\section{Problem background}\label{sec2}

We consider thermal conductivity recovery from sparse temperature measurements.
For conciseness, all mentions of conductivity in the following context refer to thermal conductivity.
The underlying dynamics of temperature follows
\begin{equation}\label{eqn:heat-dynamics}
    u_t - \nabla\cdot[a(\mathbf{x}) \nabla u ] = f(t,\mathbf{x}),
\end{equation}
subject to periodic boundary conditions.
Here, \(a:\mathcal{D}\to\mathbb{R}^+\) denotes the conductivity, which is constant in both time and temperature \citep{Zheng2011}, and \(u:\mathbb{R}^+\times\mathcal{D}\ni(t,\mathbf{x})\to\mathbb{R}\) denotes the temperature.
The initial condition \(u(t=0,\mathbf{x})\) on the domain is given.
For simplicity, throughout this work we consider the spatial domain \(\mathcal{D}=\mathbb{T}=\mathbb{R}/\mathbb{Z}\) in the 1D case or \(\mathcal{D}=\mathbb{T}^2=\mathbb{R}^2/\mathbb{Z}^2\) in the 2D case, with the domain represented as \([0,1)\) or \([0,1)^2\), respectively.
We restrict our focus to \emph{periodic boundary conditions}, leaving other boundary conditions such as Dirichlet or Neumann for future consideration.
In the 1D case, the gradient operator is \(\nabla=\partial_x\), while in the 2D case, \(\nabla=\begin{bmatrix}\partial_x & \partial_y\end{bmatrix}^\top\), denoting the spatial derivatives.

\subsection{Recoverability}
We first rule out the scenarios where it is impossible to recover the conductivity from temperature measurements due to the lack of uniqueness.
We propose the following criteria where even full observation of temperature fails to recover any information about \(a(\mathbf{x})\).
\begin{proposition}[Non-recoverability of conductivity]\label{prop:non-recoverability}
Consider the temperature dynamics governed by \eqref{eqn:heat-dynamics} with the initial condition \(u_0(\mathbf{x}):=u(t=0, \mathbf{x})\) provided.
If the following conditions hold:
\begin{itemize}
    \item \(\nabla u(t=0, \mathbf{x}) = 0 \quad \forall \mathbf{x} \in \mathcal{D}\),
    \item \(\nabla f(t, \mathbf{x}) = 0 \quad \forall (t, \mathbf{x}) \in \mathbb{R}^+ \times \mathcal{D}\),
\end{itemize}
then the conductivity \(a(\mathbf{x})\) cannot be reconstructed from \(\{u(t,\mathbf{x})\}_{\mathbb{R^+}\times\mathcal{D}}\), which is the full temperature observation.
\end{proposition}
\begin{proof}
Given \(\nabla u(t=0, \mathbf{x}) = 0\) and \(\nabla f(t, \mathbf{x}) = 0\), \eqref{eqn:heat-dynamics} simplifies to:
\[ u_t = f(t), \qquad \text{where } u(t,\mathbf{x}) = u(t) \text{ and } f(t, \mathbf{x})=f(t), \quad \forall (t,\mathbf{x})\in\mathbb{R}^+\times\mathcal{D}, \]
which is a scalar ordinary differential equation (ODE) in time, independent of \(\mathbf{x}\).
The solution to this ODE is:
\[ u(t) = u(0) + \int_0^t f(s) \, \mathrm{d}s. \]
Since \( a(\mathbf{x}) \) does not appear in this solution, it cannot be reconstructed from the temperature observations.
\end{proof}
\begin{remark}
The condition \(\nabla u(t=0, \mathbf{x}) = 0\) indicates that the initial temperature distribution is spatially constant, preventing any heat flow between locations. 
Additionally, \(\nabla f(t, \mathbf{x}) = 0\) ensures that the external heat source is uniform across the domain, so the temperature distribution remains spatially constant throughout the time.
These conditions lead to a degenerate scenario where \eqref{eqn:heat-dynamics} reduces to an ODE, and the spatial variation of conductivity \(a(\mathbf{x})\) becomes irrelevant, resulting in the loss of uniqueness in the inverse problem.
\end{remark}

\begin{lemma}[Lipschitz dependence of solution on conductivity]
\label{lemma:lipschitz-dep}
Let domain be \(\mathcal{D}=\mathbb{T}\) or \(\mathcal{D}=\mathbb{T}^2\).
Let \(a,b\in{C}^\infty(\mathcal{D};\mathbb{R}^+)\) be two bounded non-degenerate conductivities:
\[ m:=\min\left\{\inf_\mathcal{D} a, \inf_\mathcal{D} b \right\}>0, \qquad M:=\sup_\mathcal{D} \{a+b\} <\infty. \]
Let \(u(t,\mathbf{x};a), u(t,\mathbf{x};b)\) denote the solution to \eqref{eqn:heat-dynamics} with the same inhomogeneous term \(f\) and the same initial condition 
\[ u_0(\cdot):=u(t=0, \cdot; a)=u(t=0, \cdot; b)\in C^{\infty}(\mathcal{D};\mathbb{R}). \]
Then with a given final time \(T>0\),
\[ \|u(t,\mathbf{x};a) - u(t,\mathbf{x};b)\|_{L_t^2 L_\mathbf{x}^2([0,T]\times\mathcal{D})} \lesssim \|a-b\|_{L^{\infty}(\mathcal{D})}. \]
\end{lemma}
\begin{proof}
If not, we can find a sequence by taking \(\epsilon=\frac{1}{k}\) for \(k\) starting at sufficiently large (\(k\geq K\)),
\begin{multline*}
    \exists a_{k},b_{k}\in C^\infty(\mathcal{D};\mathbb{R}^+) \quad\text{with}\quad \|a_k - b_k \|_{L^\infty} \leq \frac{1}{k} \\
    \text{s.t.}\quad \|u(t,\mathbf{x};a_k) - u(t,\mathbf{x};b_k)\|_{L_t^2 L_\mathbf{x}^2([0,T]\times\mathcal{D})}\geq 1
\end{multline*}
with the following convergence implied by the Aubin--Lions compactness lemma (Theorem II.5.16 of \citet{boyer-fabrie2013}):
\begin{align*}
    a_k &\xrightarrow[k\to\infty]{L^\infty} a_{0}  &\implies u(\cdot,\cdot;a_k) &\xrightarrow[k\to\infty]{L_t^2L_\mathbf{x}^2([0,T]\times\mathcal{D})} u(\cdot,\cdot;a_0), \\
    b_k &\xrightarrow[k\to\infty]{L^\infty} b_{0}=a_{0} &\implies u(\cdot,\cdot;b_k) &\xrightarrow[k\to\infty]{L_t^2L_\mathbf{x}^2([0,T]\times\mathcal{D})} u(\cdot,\cdot;b_0).
\end{align*}
Then we have a sequence:
\[ \| u(\cdot,\cdot;a_k) - u(\cdot,\cdot;b_k) \|_{L_t^2L_\mathbf{x}^2([0,T]\times\mathcal{D})}\geq1 \qquad \forall k\geq K. \]
But this contradicts to the uniqueness of the solution to \eqref{eqn:heat-dynamics}.
Then we can take an interpolation between \(a\) and \(b\):
\[ c_n = \frac{K+1-n}{K+1} a + \frac{n}{K+1} b \qquad n=1,\cdots,K.  \]
The interpolation satisfies \(\|a-c_1\|\leq\frac{1}{K}\) and  \(\|c_K-b\|\leq\frac{1}{K}\).
Then using triangle inequalities with norms on \(L_t^2L_\mathbf{x}^2([0,T]\times\mathcal{D})\) abbreviated,
\begin{align*}
\| u(\cdot,\cdot;a) - u(\cdot,\cdot;b)\| =
\|u(\cdot,\cdot;a) - u(\cdot,\cdot;c_1) + u(\cdot,\cdot;c_1) - \cdots + u(\cdot,\cdot;c_{K}) - u(\cdot,\cdot;b)\| \\
\leq \|u(\cdot,\cdot;a) - u(\cdot,\cdot;c_1)\| + \|u(\cdot,\cdot;c_1) - u(\cdot,\cdot;c_2)\|  + \cdots + \|u(\cdot,\cdot;c_{K}) - u(\cdot,\cdot;b)\|
\end{align*}
Each component is bounded by \(\|a-c_1\|, \|c_1-c_2\|, \cdots,\|c_K-b\|\), with norms on \(L^\infty(\mathcal{D})\) abbreviated.
\end{proof}
\begin{remark}
\(m>0\) is a critical condition.
Otherwise, the temperature will never change from \(u_0\) where the conductivity is zero, and will eventually change even if conductivity is slightly larger than zero.
\end{remark}
\begin{remark}
The Lipschitz dependence of the solution on the conductivity ensures numerical stability: small perturbations in the non-degenerate conductivity will result in only minor changes to the temperature dynamics.
\end{remark}

\subsection{1D (special case)}
We semi-discretize \(u(t,x), a(x), f(t,x)\) using the finite difference method (FDM) which is first-order accuracy in space, to solve the forward PDE problem:
\[ u_j(t) := u(t, (j-1)\Delta x), \quad a_j := a((j-1)\Delta x), \quad f_j(t):=f(t,(j-1)\Delta x) \]
for \(j=1, 2, \cdots, J\).
Here, semi-discretize means the discretization is in space \(x\) only, not in time \(t\).
The discretization of the circle is a standard procedure where the \(x_1\) and \(x_J\) are neighbors of each other (see \cref{fig:stencil-circle}).
\begin{figure}[htbp]
\begin{subfigure}[b]{0.4\textwidth}
\begin{center}
\begin{picture}(50,50)
\put(25,25){\circle{40}}
% Dot 1 (0 degrees, right)
\put(45,25){\circle*{2}}
\put(46,24){\(\scriptstyle x_1\)}

% Dot 2 (60 degrees, upper-right)
\put(42,35){\circle*{2}}
\put(44,33){\(\scriptstyle x_2\)}

% Dot 3 (120 degrees, upper-left)
\put(35,42){\circle*{2}}
\put(35,44){\(\scriptstyle x_3\)}

\put(25,45){\circle*{2}}
\put(21,46){\(\scriptstyle\cdots\)}
\put(25,5){\circle*{2}}
\put(21,0){\(\scriptstyle\cdots\)}

\put(35,8){\circle*{2}}
\put(38,6){\(\scriptstyle x_{J-1}\)}

\put(42,15){\circle*{2}}
\put(45,15){\(\scriptstyle x_J\)}
\end{picture}
\end{center}
\caption{\(\mathbb{T}\sim[0,1)\)}
\label{fig:stencil-circle}
\end{subfigure}
\begin{subfigure}[b]{0.4\textwidth}
\begin{center}
\begin{picture}(60,60)  % Adjust the size to the useful area
% Draw the unit square
\put(0,0){\line(1,0){50}}  % Bottom side
\put(0,0){\line(0,1){50}}  % Left side
\multiput(50,0)(0,3){17}{\line(0,1){2}} % Right side (dashed)
\multiput(0,50)(3,0){17}{\line(1,0){2}} % Top side (dashed)
% Draw x and y axes
\put(0,0){\vector(1,0){60}} % x-axis
\put(0,0){\vector(0,1){60}} % y-axis
% Label axes
\put(55,5){\(x\)}
\put(5,55){\(y\)}
% Place a grid of black dots inside the unit square
\multiput(0,0)(10,0){5}{\multiput(0,0)(0,10){5}{\circle*{2}}}
% Add \Delta x between two horizontal dots near the origin
\put(10,6){\tiny\(\Delta x\)}
% Brace under \Delta x
\put(10,0){\tiny\(\overbrace{\hspace{3pt}}\)}
% Add \Delta y between two vertical dots near the origin
\put(4,14){\tiny\(\Delta y\)}
% Brace to the left of \Delta y
\put(0,14){\tiny\(\left.\rule{0pt}{5pt}\right\}\)}
\end{picture}
\end{center}
\caption{\(\mathbb{T}^2\sim[0,1)^2\)}
\label{fig:stencil-torus}
\end{subfigure}
\caption{Parametrization of \(\mathcal{D}\) for the finite difference method
% \KL{can you center the figures?}
}
\label{fig:stencil}
\end{figure}

A finite difference scheme of \(\partial_x(a(x)u_x)\) in energy-conservative form is:
\begin{align*}
    \partial_x ( a(x) u_x ) &\approx D_x^- (a_j D_x^+ u_j) = D_x^- \left( a_j \cdot \frac{u_{j+1}-u_j}{\Delta x} \right) \\
    &= \frac{1}{\Delta x}\left( a_j \cdot \frac{u_{j+1}-u_j}{\Delta x} - a_{j-1}\cdot \frac{u_{j}-u_{j-1}}{\Delta x} \right) \\
    &= \frac{a_j(u_{j+1}-u_j) - a_{j-1}(u_j - u_{j-1}) }{\Delta x^2}.
\end{align*}
Then the PDE becomes a system of ODEs:
\[ \dot{u}_j(t) = \frac{1}{\Delta x^2} [ a_j u_{j+1}(t) - (a_j + a_{j-1})u_{j}(t) + a_{j-1} u_{j-1}(t)] + f_j(t), \quad j=1,\cdots,J.  \]
We can write the system in matrix form:
\begin{equation}\label{eqn:forward-ODE1}
\dot{\mathbf{u}}(t) = A(a) \mathbf{u}(t) + \mathbf{f}(t),
\end{equation}
where
\begin{align*}
\mathbf{u}(t)&=\begin{bmatrix} u_1 & u_2 & \cdots & u_{J-1} & u_{J} \end{bmatrix}^\top(t), \qquad
\mathbf{f}(t)=\begin{bmatrix} f_1 & f_2 & \cdots & f_{J-1} & f_{J} \end{bmatrix}^\top(t)\\
A(a)&=\frac{1}{\Delta x^2}
\begin{bmatrix} -(a_1 + a_J) & a_1 & && a_J  \\
a_1 & -(a_2+a_1) & a_2 \\
& \ddots & \ddots & \ddots \\
& & a_{J-2} & -(a_{J-1} + a_{J-2}) & a_{J-1} \\
a_J &&& a_{J-1} & -(a_J + a_{J-1})
\end{bmatrix}_{J\times J}
\end{align*}
By Duhamel's formula, the solution can be written as
\begin{equation}\label{eqn:forward-solver}
\mathbf{u}(t;a) = \exp(tA(a))\mathbf{u}(0) + \int_{0}^{t}\exp((t-\tau)A(a))\mathbf{f}(\tau)\mathrm{d}\tau,
\end{equation}
where \(\mathbf{u}(0) = \begin{bmatrix} u(t=0,x_1) & u(t=0, x_2) & \cdots & u(t=0,x_{J-1}) & u(t=0, x_{J})\end{bmatrix}^\top\).
This discrete initial condition can be obtained from the given initial condition evaluated on the discretized grid.
We split the variables and parameters of a function using semicolon so that \(\mathbf{u}(t;a)\) treats \(t\) as the variable and conductivity \(a\) as parameters.

The motion of the sensor is parametrized by \(s:\mathbb{R}^+\to\mathbb{T}\), and the time indices it record temperature is \(t_m\) for \(m=1,\cdots,M\), with convention that
\[ 0 < t_1 < t_2 < \cdots < t_M. \]
The recorded temperature at each time \(t_m\) is \(u(t_m, s(t_m))\).
Then we propose the following least square minimizer as the reconstructed conductivity:
\begin{equation}\label{eqn:loss-pointwise}
a^*=\underset{a\in L_+^2(\mathbb{T})}{\arg\min} \sum_{m=1}^{M} [ u(t_m, s(t_m)) - \texttt{Extract}_{t_m, s(t_m)}(\tilde{u}(\cdot, \cdot; a)) ]^2,
\end{equation}
where \(\tilde{u}(\cdot, \cdot; a)\) is the simulated temperature dynamics on \((t,x)\in[0,t_M]\times\mathbb{T}\) using the candidate conductivity \(a\).
We use the nearest neighbor on the grid to represent \(s(t_m)\),
\[  j_m:= \underset{j\in\{1,2,\cdots,J\}}{\arg \min} | x_j - s(t_m) | \implies x_{j_m}\approx s(t_m), \qquad m=1,2,\cdots, M. \]
Let \(\mathbf{e}_j\) denote the \(j^\text{th}\) standard basis for \(\mathbb{R}^{J}\) (one-hot vector with 1 in \(j^\text{th}\) coordinate), then we can explicitly define
\(\texttt{Extract}_{t_m, s(t_m)}(\tilde{u}(\cdot, \cdot; a)) := \mathbf{e}_{j_m}^\top \mathbf{u}(t_m; a)\).
\begin{remark}
Although the formulation above is expressed as a least-square minimization, it is crucial to recognize that reconstructing the conductivity \(a\) (or \(a_1,\cdots,a_J\) after discretization) is not a linear least-square problem.
This distinction arises because the governing PDE \eqref{eqn:heat-dynamics} is linear with respect to the solution \(u\), but it is nonlinear with respect to the parameter \(a\), which can be directly seen from \eqref{eqn:forward-solver}.
Consequently, standard linear least-square techniques are not directly applicable, and specialized numerical methods must be employed to solve for the optimal \(a^*\).
\end{remark}

\subsection{Sensitivity analysis}
Based on the spectral analysis of \(A(a)\)\footnote{For simplicity, we omit the explicit reference to the \(a\)-dependence of \(A\) for the remainder of the 1D problem.}, measurements taken at the same location over a later time become less sensitive to reconstruction errors, given that the external heat source is spatially constant.
Formally, we propose the following:
\begin{theorem}
Let the eigenvalue decomposition of \(A(a)\) in \eqref{eqn:forward-ODE1} (which is Hermitian) be
\begin{align*}
    A = \frac{1}{\Delta x^2} Q \Lambda Q^\top,
\quad \text{where} \quad Q &= \begin{bmatrix} v_1 & v_2 & \cdots & v_{J-1} & v_{J} \end{bmatrix}_{J\times J}, \\
\quad \text{and} \quad \Lambda &= \mathtt{diag}(\lambda_1,\cdots,\lambda_J),
\end{align*}
with the eigenvalues sorted in ascending order:
\[ \lambda_1 \leq \lambda_2 \leq \cdots \leq \lambda_{J-1} < \lambda_J = 0. \]
Suppose the external heat source is spatially constant, i.e., \(f(t,x) = f(t)\) for all \((t,x) \in \mathbb{R}^+ \times \mathbb{T}\).
If two measurements, \(u(t_m, s(t_m))\) and \(u(t_n, s(t_n))\), where \(1 \leq m < n \leq M\), are taken at the same position but at different times, namely
\[ t_n > t_m > \frac{\Delta x^2}{\sqrt{-\lambda_k}}, \quad s(t_n) = s(t_m), \quad \text{for some } k = 1, 2, \cdots, J-1, \]
then the measurement \(u(t_m, s(t_m))\) is more sensitive to perturbations in \(\lambda_k\) than \(u(t_n, s(t_n))\).
\end{theorem}
\begin{proof}
Notice that the row sums of \(A\) are zero\footnote{This is also why we chose periodic boundary conditions throughout this paper and leave other boundary conditions for future work.}, which gives the dominant eigen-pair \(\lambda_J = 0, v_J = (1/\sqrt{J}) \mathbf{1}_J\), where \(\mathbf{1}_J \in \mathbb{R}^{J}\) denotes the all-one vector.
This eigen-pair corresponds to the constant temperature being a solution to the spatially constant external heat source version of \eqref{eqn:heat-dynamics} with constant initial data.
Hence, \eqref{eqn:forward-solver} becomes
\[ \mathbf{u}(t; a) = \exp(t A) \mathbf{u}(0) + \int_{0}^{t} e^{t-\tau} f(\tau)\mathrm{d}\tau \cdot \mathbf{1}_J. \]
Using spectral decomposition on \(\exp(t_m A)\),
\begin{align*}
    \exp(t_m A) &= Q ~ \mathtt{diag}\left(\exp\left[\frac{t_m}{\Delta x^2} \lambda_1 \right], \cdots, \exp\left[\frac{t_m}{\Delta x^2} \lambda_J \right]\right) Q^\top
    = \sum_{j=1}^{J} \exp\left(\frac{t_m}{\Delta x^2} \lambda_j\right) v_j v_j^\top \mathbf{u}(0).
\end{align*}
Measurement \(u(t_m, s(t_m))\) can be extracted from \(\mathbf{u}(t_m; a)\) by the dot product with \(\mathbf{e}_{j_m}\):
\[ u(t_m, s(t_m)) = \mathbf{e}_{j_m}^\top \sum_{j=1}^{J} \exp\left(\frac{t_m}{\Delta x^2} \lambda_j\right) v_j v_j^\top \mathbf{u}(0) + \int_{0}^{t_m} e^{t_m-\tau} f(\tau) \, \mathrm{d}\tau. \]
We take derivatives with respect to \(\lambda_k\) to analyze the sensitivity:
\begin{align*}
    \frac{\partial u(t_m, s(t_m))}{\partial \lambda_k} &= \frac{\partial}{\partial \lambda_k} \exp\left(\frac{t_m}{\Delta x^2} \lambda_k\right) \mathbf{e}_{j_m}^\top v_k v_k^\top \mathbf{u}(0)\\
    &= \frac{t_m}{\Delta x^2} \exp\left(\frac{t_m}{\Delta x^2} \lambda_k\right) \mathbf{e}_{j_m}^\top v_k v_k^\top \mathbf{u}(0).
\end{align*}
Similarly,
\[ \frac{\partial u(t_n, s(t_n))}{\partial \lambda_k} = \frac{t_n}{\Delta x^2} \exp\left(\frac{t_n}{\Delta x^2} \lambda_k\right) \mathbf{e}_{j_n}^\top v_k v_k^\top \mathbf{u}(0), \]
and since the position extraction vectors are the same, \textit{i.e.}, \(\mathbf{e}_{j_m} = \mathbf{e}_{j_n}\), by \(s(t_m) = s(t_n)\), it suffices to compare the coefficients:
\[ \frac{t_m}{\Delta x^2} \exp\left(\frac{t_m}{\Delta x^2} \lambda_k\right) > \frac{t_n}{\Delta x^2} \exp\left(\frac{t_n}{\Delta x^2} \lambda_k\right), \]
which is clear from the monotonically decreasing region of the rescaled function \(x\mapsto x\exp(\lambda x)\):
\[ \frac{\mathrm{d}}{\mathrm{d}t} \left[ \frac{t}{\Delta x^2} \exp\left(\frac{t}{\Delta x^2} \lambda_k\right) \right] = \frac{1}{\Delta x^2} \left( 1 + \lambda_k \frac{t^2}{\Delta x^2} \right) \exp\left(\frac{t}{\Delta x^2} \lambda_k\right) < 0, \quad \forall t > \frac{\Delta x^2}{\sqrt{-\lambda_k}}. \]
\end{proof}
\begin{remark}
From the calculation above, we also conclude that when \(t_m\gg1\),
\[ u(t_m, s(t_m)) - \underbrace{\frac{1}{J}\mathbf{1}_J^\top \mathbf{u}(0)}_\text{average} - \underbrace{\int_{0}^{t_m} e^{t_m-\tau} f(\tau) \, \mathrm{d}\tau}_\text{external heat}  = \sum_{j=1}^{J-1} \underbrace{\exp\left(\frac{t_m}{\Delta x^2} \lambda_j\right)}_{\approx0} \mathbf{e}_{j_m}^\top v_j v_j^\top \mathbf{u}(0). \]
\end{remark}
\begin{remark}\label{rmk:sensitivity}
The sensitivity can be understood as follows: if there is a small perturbation in the reconstructed conductivity from the true conductivity \(a\) (in terms of the eigenvalues of \(A(a)\)), then the calculated temperature based on the perturbed conductivity will significantly deviate from the actual measurement.
We can see from the proof that the sensitivity with respect to time goes up and down and acts like the function \(t\mapsto  t\exp(\lambda_k t)\).
This implies that if the measurement is taken too early, the temperature is primarily influenced by the initial condition.
Conversely, if the measurement is taken too late, the temperature is dominated by the average temperature and the accumulated effect of the external heat source (i.e., equilibrium).
In both cases, the temperature is mainly driven by the initial condition, which is independent of the conductivity. 
Therefore, there exists a time interval during which the temperature measurement is highly sensitive to errors in the reconstruction, making this period crucial for accurate conductivity estimation.
Moreover, this interval is shorter for smaller \(\lambda_k\).
In practice, if the measurements are taken too early for \(\lambda_k\)'s, we can always continue recording until the system enters the highly sensitive interval.
However, if the measurement is taken too late for \(\lambda_{J-1}\), then it's also too late for all other smaller \(\lambda_k\)'s, which can be treated as garbage data, and the reconstruction of conductivity is inevitably prone to large errors.
\end{remark}

\subsection{Extension to 2D space}
Consider \eqref{eqn:heat-dynamics} on \(\mathbb{T}^2=\mathbb{R}^2/\mathbb{Z}^2\sim[0, 1)^2\) with \(\mathbf{x}=(x,y)\):
\begin{equation}\label{eqn:heat-dynamics-2D}
    u_t(t,x,y) = \partial_x[a(x,y) \partial_x u] + \partial_y[a(x,y) \partial_y u] + f(t,x,y).
\end{equation}
Again, \(u(t,x,y)\) is the temperature at time \(t\) and spatial position \((x,y)\) and \(a(x,y)\) is the surface thermal conductivity.
We inherit the 1D convention that \(x_i=(i-1)\Delta x\) and \(y_j=(j-1)\Delta y\) for \(i\in\{1,2,\cdots,K\}\) and \(j\in\{1,2,\cdots,J\}\), where \(K\Delta x=1\) and \(J\Delta y=1\) (so that they are grid points on a uniform grid on the unit square).
Similar as 1D case, we semi-discretize the spatial domain according to \cref{fig:stencil-torus} into
\[ u_{i,j}(t):= u(t, x_i, y_j), \qquad a_{i,j}:=a(x_i,y_j), \qquad f_{i,j}(t):=f(t, x_i, y_j). \]
We will discretize the right-hand side of \eqref{eqn:heat-dynamics-2D} as follows \footnote{Recall that for economy of notation we drop explicit dependency of \(u_{i,j}\) and \(f_{i,j}\) on \(t\)}:
\begin{align*}
D_x^-\left(a_{i,j}D_x^+ u_{i,j}\right)+D_y^-\left(a_{i,j}D_y^+ u_{i,j}\right)
&\approx\frac{1}{\Delta x}\left[a_{i,j}\cdot \frac{u_{i+1,j} - u_{i,j}}{\Delta x} - a_{i-1,j}\cdot\frac{u_{i,j}-u_{i-1,j}}{\Delta x} \right] \\
&+ \frac{1}{\Delta y}\left[a_{i,j}\cdot \frac{u_{i,j+1} - u_{i,j}}{\Delta y} - a_{i,j-1}\cdot\frac{u_{i,j}-u_{i,j-1}}{\Delta y} \right].
\end{align*}
For simplicity, we take \(\Delta x=\Delta y\), which implies \(K=J\), leading to 
\begin{multline*}
\nabla\cdot[a(x_i,y_j)\nabla u(t,x_i,y_j)]\approx
\frac{1}{\Delta x^2}\left[ a_{i,j}(u_{i+1,j}-u_{i,j}) - a_{i-1,j}(u_{i,j}-u_{i-1,j})\right.\\
\left. + a_{i,j}(u_{i,j+1} - u_{i,j}) -a_{i,j-1}(u_{i,j} - u_{i,j-1}) \right].
\end{multline*}
Then we have a system of ODEs (with \(i,j\in\{1,2,\cdots,J\}\)):
\begin{multline}\label{eqn:forward-ODE2}
\dot{u}_{i,j} = \frac{1}{\Delta x^2}[ a_{i,j} u_{i+1,j} + a_{i-1,j} u_{i-1,j} + a_{i,j} u_{i,j+1} \\
+ a_{i,j-1} u_{i,j-1} -(2a_{i,j} + a_{i-1,j} + a_{i,j-1}) u_{i,j} ] + f_{i,j}.
\end{multline}
This system is of the form of \eqref{eqn:forward-ODE1} where \(A(a)\in\mathbb{R}^{J^2\times J^2}\).
The unique solution can again be obtained by Duhamel's formula, and we denote it as \(u(\cdot,\cdot;a)\).
Then we consider the observation of a sensor whose position at time \(t\) is \(s(t)\in\mathbb{T}^2\).
Therefore, the recorded temperature by this sensor is \(\{u(t_m, s(t_m))\}_{m=1,\cdots,M}\).
Similar as 1D, we need the conductivity that minimizes the difference between the recorded temperature and simulated temperature based on the candidate conductivity \(a\), so we use exactly the same minimizer of \eqref{eqn:loss-pointwise},
where the \texttt{Extract} operation is also based on nearest neighbor approximation, which implies \(s(t_m) \approx (x_{i_m}, y_{j_m})\).

\section{Sensor Motion, Initial Condition, and External Heat Source}\label{sec3}

Other than the temperature \(u\) and conductivity \(a\), the terms in the PDE need to be considered are:
\begin{itemize}
    \item Sensor motion: \(s(t)\);
    \item Initial condition \(u(t=0,\cdot)\);
    \item External heat source \(f(t,\cdot)\).
\end{itemize}

\subsection{Design of Sensor Motion}
The sensitivity analysis above provides insight into the advantages of using moving sensors over static ones.
While static sensors are limited to detecting reconstruction errors in conductivity within their immediate vicinity, moving sensors are able to detect errors across the entire region they traverse.
As mentioned in Remark~\ref{rmk:sensitivity}, there is a specific time interval during which temperature measurements are highly sensitive to reconstruction errors in conductivity.
During this critical period, the more of the spatial domain the sensors traverse, the better the reconstruction quality.
To fully exploit this sensitivity window, it is advantageous for moving sensors to cover the entire spatial domain.
This approach ensures that sensors gather comprehensive data from across the domain during the most sensitive phase, thereby improving the precision and accuracy of the reconstruction.

\subsubsection{Sensor Placement and Movements for \texorpdfstring{\(\mathcal{D} = \mathbb{T}\)}{D = T}}
In one-dimensional settings, a single sensor can effectively traverse the entire domain by moving along \(\mathbb{T}\) at a constant speed over the unit time interval \(t \in [0,1]\) (by setting \(t_M=1\)).
This approach ensures ergodicity, as the sensor returns to its initial position after covering the entire domain, thereby maximizing its ability to capture detailed temperature data within the given timeframe.

\subsubsection{Sensor Placement and Movements for \texorpdfstring{\(\mathcal{D} = \mathbb{T}^2\)}{D = T2}}\label{subsubsec}
Extending the methodology of using a single sensor to traverse the domain to higher dimensions presents significant challenges.
In dimensions greater than one, smooth curves occupy zero measure \citep{bams/1183504867}, meaning that a sensor following such a trajectory would collect temperature data from an infinitesimally small portion of the domain.
Attempts to traverse space-filling curves (or their finite fractal levels) within the high-sensitivity window leads to infinite (or large) moving speeds of the sensor, which leads to physical infeasibility.
Consequently, the data gathered would be insufficient for accurately reconstructing the conductivity across the entire spatial region, limiting the reconstruction accuracy to only a localized area.

Given these constraints, the physical movement of sensors must prioritize continuity and regularity, often restricting trajectories to piecewise constant functions with bounded jumps after spatial discretization.
This means that the sensor either moves between adjacent pixels or stays in place, i.e.,
\[ |s(t_{m}) - s(t_{m-1})|_{\text{Manhattan on discretized }\mathbb{T}^2} \leq \Delta x, \quad m = 1, \cdots, M. \]
This limitation makes it impossible for a single sensor to cover the entire discretized domain \(\mathbb{T}^2\) within a reasonable timeframe.
Furthermore, temperature measurements taken at later times tend to contribute less effectively to the reconstruction of the conductivity.

To overcome these limitations, deploying multiple temperature sensors becomes essential.
By distributing sensors across different locations and ensuring their trajectories cover reasonably large regions of the domain, broader spatial coverage is achieved early in the measurement process.
This strategic deployment enhances the accuracy of thermal conductivity reconstruction and mitigates the inefficiencies associated with single-sensor approaches.
Therefore, utilizing multiple sensors with well-planned trajectories is crucial for comprehensive and reliable data collection in higher-dimensional domains.

\begin{figure}[htbp]
\begin{center}
\begin{picture}(100,100)  % Adjust the size to the useful area
% Draw the unit square
\put(0,0){\line(1,0){80}}  % Bottom side
\put(0,0){\line(0,1){80}}  % Left side

\put(10,90){\circle*{4}} % Filled circle
\put(10,90){\vector(0,-1){8}}
\multiput(10,10)(0,4){20}{\line(0,1){2}} % Left side (dashed)
\multiput(10,10)(4,0){20}{\line(1,0){2}} % Bottom side (dashed)
\multiput(90,10)(0,4){20}{\line(0,1){2}} % Right side (dashed)
\multiput(10,90)(4,0){20}{\line(1,0){2}} % Top side (dashed)

\put(17,75){*}
\put(20,80){\vector(0,-1){12}}
\multiput(20,20)(0,4){15}{\line(0,1){2}} % Left side (dashed)
\multiput(20,20)(4,0){15}{\line(1,0){2}} % Bottom side (dashed)
\multiput(80,20)(0,4){15}{\line(0,1){2}} % Right side (dashed)
\multiput(20,80)(4,0){15}{\line(1,0){2}} % Top side (dashed)

\put(28,68){\rule{4pt}{4pt}} % Filled square
\put(30,70){\vector(0,-1){16}}
\multiput(30,30)(0,4){10}{\line(0,1){2}} % Left side (dashed)
\multiput(30,30)(4,0){10}{\line(1,0){2}} % Bottom side (dashed)
\multiput(70,30)(0,4){10}{\line(0,1){2}} % Right side (dashed)
\multiput(30,70)(4,0){10}{\line(1,0){2}} % Top side (dashed)

\put(40,60){\circle{4}} % Hollow circle
\put(40,60){\vector(0,-1){16}}
\multiput(40,40)(0,4){5}{\line(0,1){2}} % Left side (dashed)
\multiput(40,40)(4,0){5}{\line(1,0){2}} % Bottom side (dashed)
\multiput(60,40)(0,4){5}{\line(0,1){2}} % Right side (dashed)
\multiput(40,60)(4,0){5}{\line(1,0){2}} % Top side (dashed)

% Draw x and y axes
\put(0,0){\vector(1,0){100}} % x-axis
\put(95,3){\(x\)}
\put(20,0){\circle*{2}}\put(15,2){\scriptsize 1/5}
\put(40,0){\circle*{2}}\put(35,2){\scriptsize 2/5}
\put(60,0){\circle*{2}}\put(55,2){\scriptsize 3/5}
\put(80,0){\circle*{2}}\put(75,2){\scriptsize 4/5}

\put(0,0){\vector(0,1){100}} % y-axis
\put(3,95){\(y\)}
\put(0,20){\circle*{2}}\put(1,18){\(\tfrac{1}{5}\)}
\put(0,40){\circle*{2}}\put(1,38){\(\tfrac{2}{5}\)}
\put(0,60){\circle*{2}}\put(1,58){\(\tfrac{3}{5}\)}
\put(0,80){\circle*{2}}\put(1,78){\(\tfrac{4}{5}\)}
\end{picture}
\end{center}
\caption{Initial positions \(s_{1}(0),\cdots,s_{4}(0)\) and trajectories of the four sensors on \(\mathbb{T}^2\)}
\label{fig:sensor-init_positions}
\end{figure}
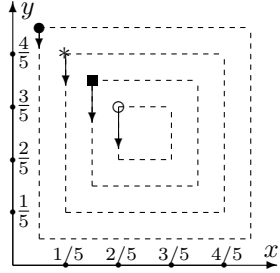
As shown in \cref{fig:sensor-init_positions}, we place four sensors at the corners of four concentric squares, each sensor moving along its respective orbit.
These squares are referred to as the \emph{orbits} of the sensors, as they always move along the sides of their initial paths.
The side lengths of these four squares are \(\tfrac{1}{5}\), \(\tfrac{2}{5}\), \(\tfrac{3}{5}\), and \(\tfrac{4}{5}\), from the innermost to the outermost square, all centered at \((x = \tfrac{1}{2}, y = \tfrac{1}{2})\).
Each sensor starts from a corner of its orbit, and the initial positions are shown as \(\raisebox{-0.3ex}{\scalebox{2}{$\cdot$}}\), *, \(\scalebox{0.8}{\(\blacksquare\)}\) and \(\circ\), respectively, along with the contours of the initial temperature.

The sensors travel counterclockwise along their orbits to the next corner, continuing this pattern until \(t = 1\).
At \(t = 1\), they have completed 8, 4, 2, and 1 cycles along their orbits, respectively, from the innermost to the outermost square, all moving at constant speeds.
By this design, we ensure that every point in \(\mathbb{T}^2\) will be approached by at least one sensor during its high-sensitivity window.

\subsection{Initial Condition and External Heat Source}
We also consider the frequency domain for the 1D spatial domain, where we take \(x \in [0,1)\) from the representative interval of \(\mathbb{T}\):
\begin{align*}
    u(t,x) &= \sum_{n \in \mathbb{Z}} \hat{\mathbf{u}}_n(t)\, e^{2\pi i n x}, &
    \hat{\mathbf{u}}_n(t) &= \int_{0}^{1} u(t,x)\, e^{-2\pi i n x}\, \mathrm{d}x, &
    \hat{\mathbf{u}}(t) &\in \ell^2(\mathbb{Z}) \quad \forall t \in [0,\infty), \\
    f(t,x) &= \sum_{n \in \mathbb{Z}} \hat{\mathbf{f}}_n(t)\, e^{2\pi i n x}, &
    \hat{\mathbf{f}}_n(t) &= \int_{0}^{1} f(t,x)\, e^{-2\pi i n x}\, \mathrm{d}x, &
    \hat{\mathbf{f}}(t) &\in \ell^2(\mathbb{Z}) \quad \forall t \in [0,\infty),\\
    a(x) &= \sum_{n \in \mathbb{Z}} \hat{a}(n)\ e^{2\pi i n x}, &
    \hat{a}(n), &= \int_{0}^{1} a(x)\, e^{-2\pi i n x}\, \mathrm{d}x &
    \hat{a} &\in \ell^2(\mathbb{Z}).
\end{align*}
Here, \(\hat{\mathbf{u}}(t)\), \(\hat{\mathbf{f}}(t)\), and \(\hat{a}\) denote the Fourier coefficients of \(u(t,\cdot)\), \(f(t,\cdot)\), and \(a\) respectively.
Substituting these expansions into Eq.~\eqref{eqn:heat-dynamics} transforms it into an infinite system of ODEs:
\[ \dot{\hat{\mathbf{u}}}_n(t) = -4\pi^2 n \sum_{m \in \mathbb{Z}} m\, \hat{a}(n - m)\, \hat{\mathbf{u}}_{m}(t) + \hat{\mathbf{f}}_{n}(t), \qquad \forall n \in \mathbb{Z}. \]
This can be written in matrix form:
\begin{equation}\label{eqn:forward-Fourier-diff}
\dot{\hat{\mathbf{u}}}(t) = -4\pi^2 D \hat{a}(Z) D\, \hat{\mathbf{u}}(t) + \hat{\mathbf{f}}(t), 
\end{equation}
where \(\hat{a}(Z)\) is a bi-infinite Toeplitz matrix\footnote{This is another reason why we prioritize periodic boundary conditions over other boundary conditions.} whose entries are given by \(\hat{a}(n - m)\) for each pair \((n, m) \in \mathbb{Z}^2\).
The diagonal matrix \(D\) and the bi-infinite Toeplitz matrix \(Z\) are defined as
\[ D = \begin{bmatrix}
    \ddots  &   &   &   &   &   & \\
            & -2 &   &   &   &   & \\
            &    & -1 &   &   &   & \\
            &    &    & 0 &   &   & \\
            &    &    &   & 1 &   & \\
            &    &    &   &   & 2 & \\
            &    &    &   &   &   & \ddots
\end{bmatrix}_{\mathbb{Z} \times \mathbb{Z}}, \qquad
Z = \begin{bmatrix}
    \ddots  & \vdots &  \vdots &  \vdots &  \vdots & \vdots & \\
    \cdots  &      0 &      -1 &      -2 &      -3 &    -4 & \cdots \\
    \cdots  &      1 &       0 &      -1 &      -2 &    -3 & \cdots \\
    \cdots  &      2 &       1 &       0 &      -1 &    -2 & \cdots \\
    \cdots  &      3 &       2 &       1 &       0 &    -1 & \cdots \\
    \cdots  &      4 &       3 &       2 &       1 &     0 & \cdots \\
            & \vdots &  \vdots &  \vdots &  \vdots & \vdots & \ddots
\end{bmatrix}_{\mathbb{Z} \times \mathbb{Z}}. \]
Again, by Duhamel's formula,
\begin{equation}\label{eqn:forward-Fourier}
\hat{\mathbf{u}}(t; a) = \exp\left[-4\pi^2 t\, D \hat{a}(Z) D \right] \hat{\mathbf{u}}(0) + \int_{0}^{t} \exp\left[-4\pi^2 (t - \tau)\, D \hat{a}(Z) D \right] \hat{\mathbf{f}}(\tau)\, \mathrm{d}\tau.
\end{equation}
Notice that \(\ker(D) = \textrm{span}\{\mathbf{e}_0\}\), which again confirms Proposition~\ref{prop:non-recoverability}.
Since we need at least one of the initial condition or the external heat source to be non-constant in space, we choose the initial condition to be non-constant while keeping the external heat source constant in space, \textit{i.e.}, \(f(t)=f(t,x)\) on \((t, x)\in[0,\infty)\times\mathbb{T}\).
This simplifies the problem to:
\[ \hat{\mathbf{u}}(t; a) = \exp\left[-4\pi^2 t\, D \hat{a}(Z) D \right] \hat{\mathbf{u}}(0) + \mathbf{e}_0\int_{0}^{t} e^{-4\pi^2 (t - \tau)} f(\tau)\, \mathrm{d}\tau . \]
As for the initial condition, we choose the lowest frequency based on the Fourier coefficients:
\[ \hat{\mathbf{u}}(0) = -\frac{i}{2}(\mathbf{e}_{1} - \mathbf{e}_{-1}) \iff u(t=0, x) = \sin(2\pi x), \qquad x\in\mathbb{T}. \]
In the 2D problem, we adopt the same criteria:
\begin{equation}\label{eqn:heat-init2}
u(t=0, x, y) = \cos(2\pi x) \cos(2\pi y),  \qquad (x,y)\in\mathbb{T}^2.
\end{equation}
As long as the external heat source is spatially constant, it neither improves nor worsens the resolution of the reconstruction.
Therefore, we choose the following form for the external heat source:
\begin{align}
    f(t,x)&= \sin(\pi t), & (t,x)\in[0,\infty)\times\mathbb{T}; \label{eqn:heat-source1} \\
    f(t,x,y)&= \sin(2\pi t), & (t,x,y)\in[0,\infty)\times\mathbb{T}^2 \label{eqn:heat-source2}.
\end{align}

\section{Overcoming Sparsity in Temperature Measurements}\label{sec4}

\subsection{Reduction of Degrees of Freedom (DoF) for \texorpdfstring{\(\mathcal{D} = \mathbb{T}\)}{D = T}}
The inverse heat equation poses significant challenges due to the non-convexity of the loss function \eqref{eqn:loss-pointwise}, rendering analytical minimization approaches often infeasible.
Additionally, the number of temperature measurements \(M\) is typically less than the desired spatial resolution \(J\) (\(M < J\)), leading to over-parameterization if all degrees of freedom are considered.
To address this, we restrict our search to low-dimensional subspaces of \(L^2(\mathbb{T})\) by parameterizing the logarithm of the conductivity \(a(\cdot)\) using Fourier series:
\[ \ln(a(x; \theta)) = \theta_0 + \sum_{k=1}^{\lfloor \dim\theta / 2 \rfloor} \theta_{2k-1} \sin(2k \pi x) + \sum_{k=1}^{\lfloor \dim\theta / 2 \rfloor} \theta_{2k} \cos(2k \pi x), \]
which is inspired from the bi-infinite Toeplitz matrix \(\hat{a}(Z)\) in \eqref{eqn:forward-Fourier}.
This parameterization allows flexibility in \(\dim\theta\), constrained by an upper bound to reflect the typical regularity of \(a(\cdot)\) or the presence of only a few abrupt transitions corresponding to material changes.
This flexibility allows us to avoid over-parameterization by setting \(\dim\theta \ll M\), even when the number of temperature measurements is limited.
Consequently, we refer to \(\theta\) as the \emph{model parameter}, as it governs the reduced-dimensional representation of the conductivity, ensuring flexibility and avoiding over-parameterization while maintaining the positivity of \(a(\cdot)\) without additional constraints, in line with the conventions in \citet{Stuart_2010} and \citet{BESKOS2017327}.

In the machine learning framework, the loss function for a candidate \(\theta\) can be derived from \eqref{eqn:loss-pointwise} as
\begin{equation}\label{eqn:loss1}
    L(\theta) := \frac{1}{M}\sum_{m=1}^{M} \left[ u(t_m, s(t_m)) - \mathbf{e}_{j_m}^\top \mathbf{u}(t_m; a(\cdot; \theta)) \right]^2,
\end{equation}
which is the standard mean-square error (MSE) between the measured temperature and calculated temperature based on reconstructed conductivity.

Notice that incrementally increasing the number of Fourier modes in the search for \(\ln(a)\)'s coefficients corresponds to retaining additional diagonals in the bi-infinite Toeplitz matrix \(\hat{a}(Z)\) as defined in \eqref{eqn:forward-Fourier}.
Specifically, a search dimension of \(\dim\theta = 2k + 1\) includes the main diagonal and its \(k\) nearest diagonals on either side, while truncating all other diagonals to zero.
For example:
\begin{itemize}
    \item \(\dim\theta = 1\): Retains only the main diagonal \(\hat{a}(0)\).
    \item \(\dim\theta = 3\): Incorporates the main diagonal and the two adjacent diagonals \(\hat{a}(0)\), \(\hat{a}(-1)\), \(\hat{a}(1)\).
    \item \(\dim\theta = 5\): Includes four nearest diagonals around the main diagonal, \(\hat{a}(0)\), \(\hat{a}(-1)\), \(\hat{a}(1)\), \(\hat{a}(-2)\), \(\hat{a}(2)\).
\end{itemize}
This systematic inclusion of diagonals allows for progressively capturing more complex features of conductivity, enhancing the reconstruction accuracy as the search dimension increases.
Thus, we propose to adaptively search the Fourier series coefficients from low frequency to high frequency; if the gradient with respect to the current coefficients nearly vanishes, we then allow higher-frequency coefficients to deviate from zero.

\begin{algorithm}
\caption{Adaptive Fourier Series Gradient Descent}\label{alg:afs}
\begin{algorithmic}[1]
\Procedure{AdaptiveFS-GD}{$N, \gamma, \epsilon, \texttt{MaxEpoch}$}
    \State \(\theta \gets (0,)\) \Comment{Initial guess \(\theta_0 = 0\) with \(\dim\theta=1\)}
    \For{\(n = 1, \ldots, \texttt{MaxEpoch}\)}
        \State \(g \gets \nabla_\theta L(\theta)\)  \Comment{\texttt{loss.backward()}}
        \If{\(\|g\|^2 < \epsilon\) and \(\dim(\theta) \leq 2N\)}
            \State \(\theta \gets (\theta, 0)\) \Comment{Include more Fourier modes}
        \Else
            \State \(\theta \gets \theta - \gamma g\) \Comment{Steepest/gradient descent}
        \EndIf
    \EndFor
    \State \textbf{return} \(\theta\)
\EndProcedure
\end{algorithmic}
\end{algorithm}

\subsection{Multiple Sensors for \texorpdfstring{\(\mathcal{D} = \mathbb{T}^2\)}{D = T2}}
In extending our approach to two-dimensional domains, the challenge of insufficient temperature measurements is mitigated by increasing the number of sensors, as demonstrated in \autoref{subsubsec}.
Consequently, the necessity to reduce DoF as in the one-dimensional case becomes less critical.
Therefore, we define the search space for the 2D case as \(\mathbb{R}^{J \times J} \simeq \mathbb{R}^{J^2}\), where each parameter corresponds to the logarithm of conductivity at discrete grid points:
\[ \theta_{(i-1)J + j} = \ln a(x_{i}, y_{j}), \quad i, j \in \{1, 2, \ldots, J\}. \]
Here, the dimension \(\dim\theta = J^2\) is fixed, and the indexing of \(\theta\) starts from 1 instead of 0, unlike in the 1D case, as there is no even/odd parity of the indices corresponding to cosine and sine basis functions.

Similar as 1D, loss is taken to be the MSE between prediction and observation:
\begin{equation}\label{eqn:loss2}
    L(\theta):= \frac{1}{MK} \sum_{k=1}^{K} \sum_{m=1}^{M} [u(t_m,s_k(t_m)) - \mathtt{Extract}_{t_m,s_k(t_m)}(\tilde{u}(\cdot,\cdot;a(\cdot;\theta))) ]^2,
\end{equation}
where \(K\) is the number of sensors and \(s_k(t)\) is the position of the \(k^\text{th}\) sensor at time \(t\).
We assume that all sensors are synchronized and record temperature measurements simultaneously.

The reason for adopting the natural pixel basis is that a naive reduction of DoF using a 2D Fourier basis does not work well in two dimensions.
A similar effect is observed in JPEG compression, where an image is segmented into \(8 \times 8\) pixel blocks, and each block is decomposed into its 2D Fourier series coefficients, instead of directly decomposing the full image into Fourier cosine series as in \citet{125072}.
More adaptive approaches are proposed and tested in \citet{Du_2023}, but we do not apply reduction of DoF in 2D due to the difficulty in implementation.

\section{Numerical Experiments}\label{sec5}

We leverage the automatic differentiation capabilities of PyTorch~\citep{Paszke2017AutomaticDI, Paszke2019PyTorchAI} and the one-step ODE solver \texttt{odeint} provided by the torchdiffeq library ~\citep{torchdiffeq}.
Automatic differentiation enables the computation of \(\nabla_\theta L(\theta)\) (via \texttt{loss.backward()} in \cref{alg:afs}) by evaluating the loss function only once, thereby requiring only a single forward PDE solve per gradient calculation.

In contrast, finite difference methods for gradient estimation necessitate multiple evaluations of the loss function.
Specifically, the gradient component in the direction of \(\mathbf{e}_j\) is approximated as:
\[ \nabla_\theta L(\theta) \cdot \mathbf{e}_j \approx \frac{1}{\Delta \theta} \left[ L(\theta + \Delta\theta\,\mathbf{e}_j) - L(\theta) \right], \qquad j=1,\ldots,\dim\theta, \]
and the full gradient \(\nabla_\theta L(\theta)\) requires \(\dim\theta + 1\) evaluations of the loss function.
This approach significantly increases computational costs, as each additional degree of freedom requires an extra forward PDE solve.
Even with parallel computation, the total number of forward PDE solves remains prohibitive.

To assess the quality of conductivity reconstructions, we define the \emph{relative error} as:
\begin{equation}\label{eqn:relative-error}
\text{Relative error} := \frac{\|a_\texttt{reconstructed} - a_\texttt{truth}\|}{\|a_\texttt{truth}\|}.
\end{equation}
For the one-dimensional case, we set \(J=100\), computing the norm in \(\mathbb{R}^{100}\).
In the two-dimensional case, we choose \(J=32\) and employ the Frobenius norm.
Both cases assume the ground truth conductivity lies within the range of \(\sim10^{-2}\), which ensures the time interval \([0,1]\) primarily overlaps with the high-sensitivity window.
If the conductivity is too large (\(>1\)), the heat dissipates too quickly, and only a small portion of the temperature measurements fall within the high-sensitivity window.
% This can also be seen in the Lipschitz constant of the forward mapping \(\mathcal{G}\): if \(a\) and \(b\) are both large, then the Lipschitz constant is large, even though \(\|a-b\|\) is small.
In such cases, the available data becomes insufficient for accurate conductivity reconstruction.
Similarly, we assume the minimum conductivity is of the same order as the maximum.
The scenario where the minimum conductivity is orders of magnitude smaller than the maximum, a typical multi-scale problem, falls outside the scope of this research and will be addressed in future work.

\subsection{\texorpdfstring{Single Sensor on \(\mathbb{T}\) with Discretization \(J=100\)}{Single sensor on T with discretization J=100}}
In the numerical experiments, we will investigate \(M=J=100\) while \(N=9\) in \cref{alg:afs}, leading to  \(\dim\theta\leq19\).
Remarkably, the code in Supplementary Materials demonstrates how the forward PDE can be efficiently solved with automatic differentiation enabled.
This allows for efficient computation of the gradient of the loss function with respect to the input parameter \(\theta\) via backpropagation \citep{Rumelhart1986}, eliminating the need for finite difference methods and significantly reducing computational overhead while improving efficiency.

\subsubsection{Testcases: Heaviside, PieceLinear3S, PieceLinear4W}

We design three different test cases to evaluate the performance of the conductivity reconstruction: Heaviside, PieceLinear3S, and PieceLinear4W.
\begin{align*}
    a^\texttt{Heaviside}(x) &= \frac{1}{100}\cdot\begin{cases} 1, & x\in\left[0,\frac{1}{2}\right) \\ 2, & x\in\left[\frac{1}{2}, 1\right) \end{cases}, \\
    a^\texttt{PieceLinear3S}(x) &= \frac{1}{100}\cdot \begin{cases}
        {2 - 3x}, & x \in \left[0, \frac{1}{3}\right) \\
        {6x - 1}, & x \in \left[\frac{1}{3}, \frac{2}{3}\right) \\
        {5 - 3x}, & x \in \left[\frac{2}{3}, 1\right) \end{cases}, \\
    a^\texttt{PieceLinear4W}(x) &= \frac{1}{100} \cdot \begin{cases}
        {2 - 4x}, & x \in \left[0, \frac{1}{4}\right) \\
        {8x - 1}, & x \in \left[\frac{1}{4}, \frac{1}{2}\right) \\
        {7 - 8x}, & x \in \left[\frac{1}{2}, \frac{3}{4}\right) \\
        {4x - 2}, & x \in \left[\frac{3}{4}, 1\right) \end{cases}.
\end{align*}
The latter two are named for their ``piecewise linear functions", with PieceLinear3S having 3 segments resembling a straightened S and PieceLinear4W having 4 segments resembling a W.
\begin{figure}[htbp]
\begin{subfigure}[h]{0.24\textwidth}
\includegraphics[width=\textwidth]{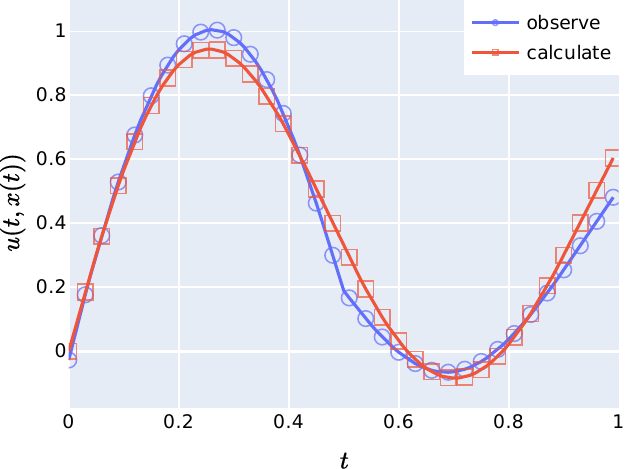}
\caption{Initial loss}\label{fig:Heaviside-initial-loss}
\end{subfigure}\hfill
\begin{subfigure}[h]{0.24\textwidth}
\includegraphics[width=\textwidth]{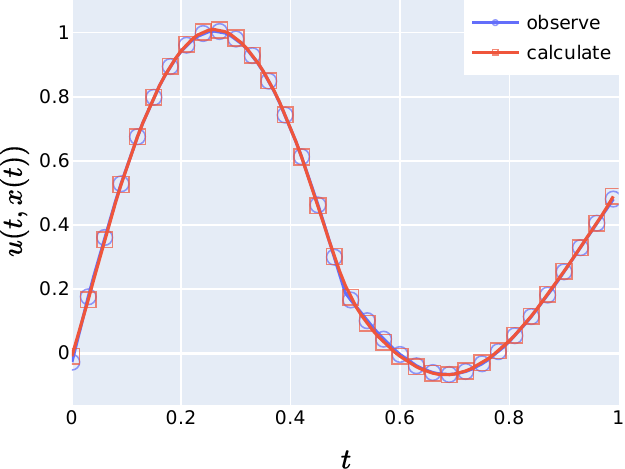}
\caption{Final loss}\label{fig:Heaviside-final-loss}
\end{subfigure}\hfill
\begin{subfigure}[h]{0.24\textwidth}
\includegraphics[width=\textwidth]{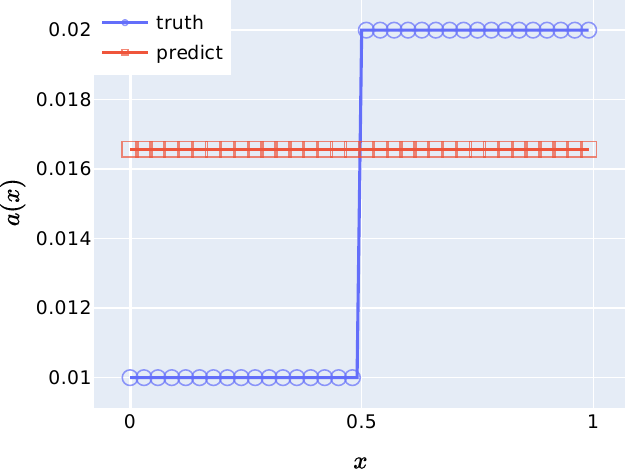}
\caption{Initial error}\label{fig:Heaviside-initial-error}
\end{subfigure}\hfill
\begin{subfigure}[h]{0.24\textwidth}
\includegraphics[width=\textwidth]{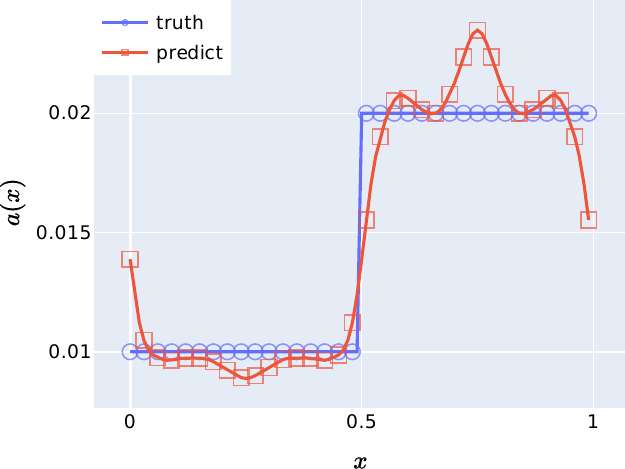}
\caption{Final error}\label{fig:Heaviside-final-error}
\end{subfigure}
\caption{1D test case: Heaviside function}
\end{figure}
\begin{figure}[htbp]
\begin{subfigure}[h]{0.24\textwidth}
\includegraphics[width=\textwidth]{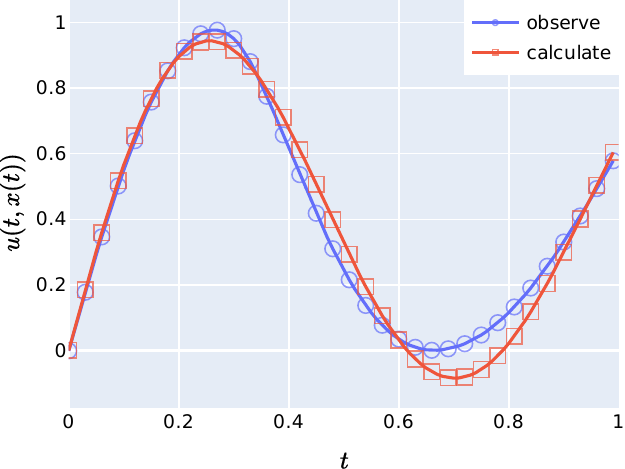}
\caption{Initial loss}
\end{subfigure}\hfill
\begin{subfigure}[h]{0.24\textwidth}
\includegraphics[width=\textwidth]{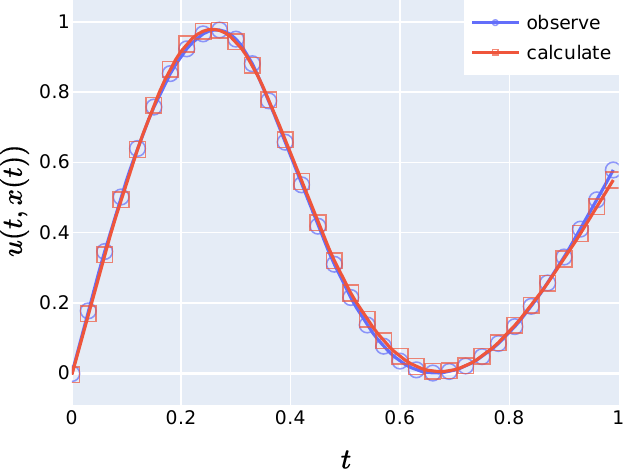}
\caption{Final loss}
\end{subfigure}\hfill
\begin{subfigure}[h]{0.24\textwidth}
\includegraphics[width=\textwidth]{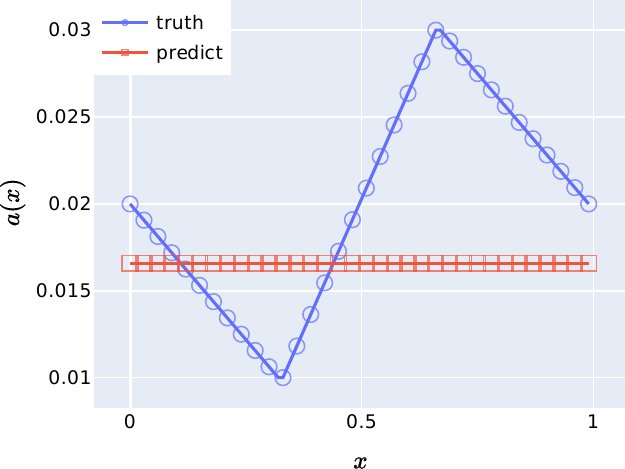}
\caption{Initial error}
\end{subfigure}\hfill
\begin{subfigure}[h]{0.24\textwidth}
\includegraphics[width=\textwidth]{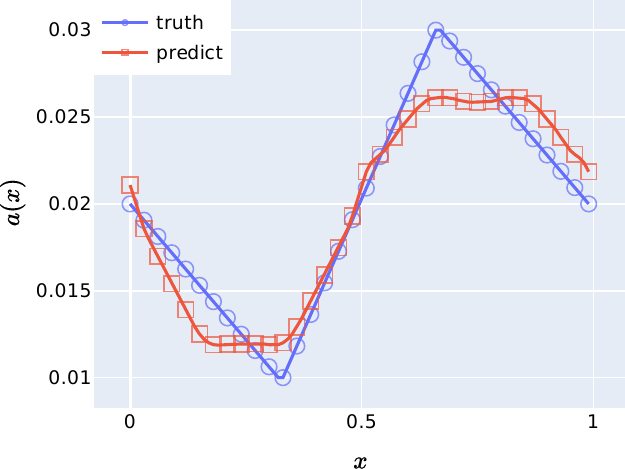}
\caption{Final error}
\end{subfigure}
\caption{1D test case: Piecewise Linear (3-piece, S shape)}
\label{fig:PieceLinear3S}
\end{figure}
\begin{figure}[htbp]
\begin{subfigure}[h]{0.24\textwidth}
\includegraphics[width=\textwidth]{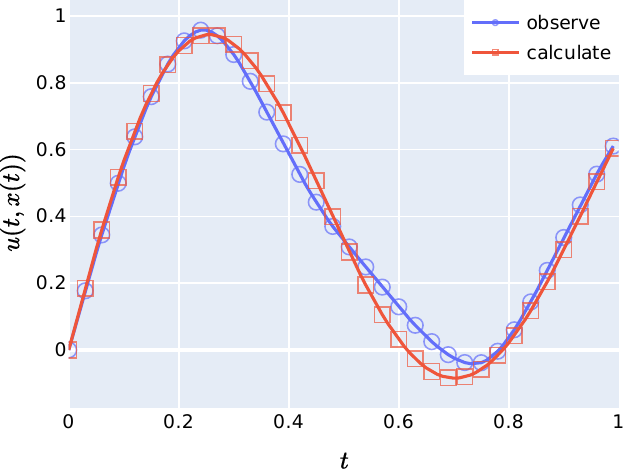}
\caption{Initial loss}
\end{subfigure}\hfill
\begin{subfigure}[h]{0.24\textwidth}
\includegraphics[width=\textwidth]{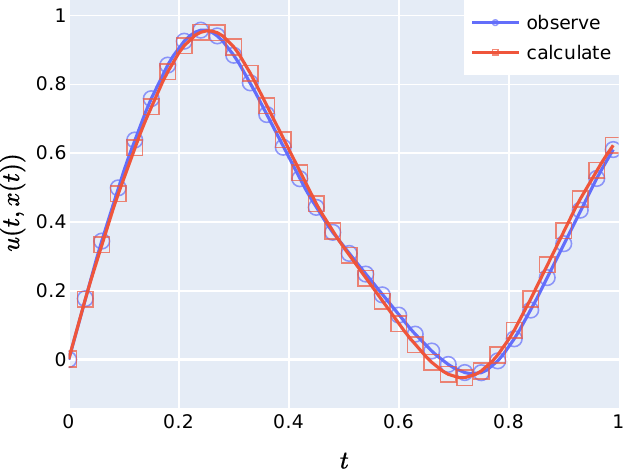}
\caption{Final loss}
\end{subfigure}\hfill
\begin{subfigure}[h]{0.24\textwidth}
\includegraphics[width=\textwidth]{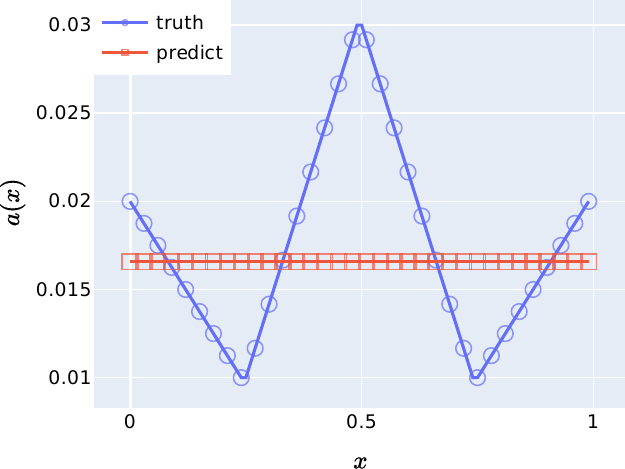}
\caption{Initial error}
\end{subfigure}\hfill
\begin{subfigure}[h]{0.24\textwidth}
\includegraphics[width=\textwidth]{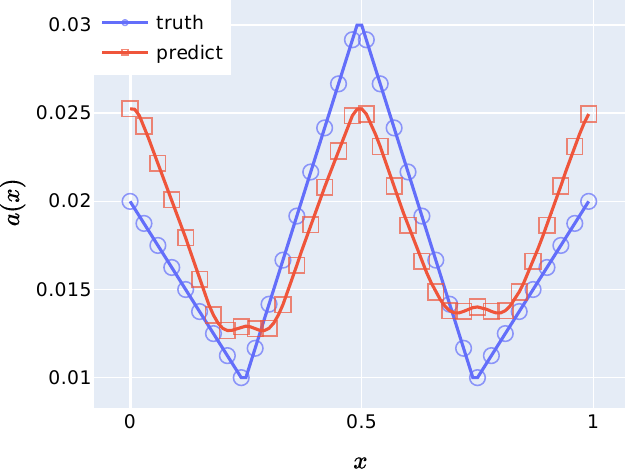}
\caption{Final error}
\end{subfigure}
\caption{1D test case: Piecewise Linear (4-piece, W shape)}
\label{fig:PieceLinear4W}
\end{figure}
The blue circled line in \cref{fig:Heaviside-initial-loss} represents the temperature measurement obtained from the sensor (by solving the forward problem), while the red squared line is the calculated temperature measurement based on the initial guess of conductivity (red squared line in \cref{fig:Heaviside-initial-error}).
The \(L^2\) difference between these two lines in \cref{fig:Heaviside-initial-loss} represents the loss, while \(L^2\) difference between the two lines in \cref{fig:Heaviside-initial-error} is the error.
After training, the calculated temperature measurement is plotted as red squared line in \cref{fig:Heaviside-final-loss}, which closely matches the temperature recorded by the sensor.
Additionally, the underlying reconstructed conductivity is plotted as the red squared line in \cref{fig:Heaviside-final-error}, showing significant improvement over the constant initial guess.
The same setup is performed on two other conductivity profiles as shown in \cref{fig:PieceLinear3S} and \cref{fig:PieceLinear4W}, which leads to similar improvements.

\subsubsection{Optimization record}
We illustrates the progress of \cref{alg:afs} through three key metrics: The number of Fourier basis functions, the loss, and the error over 500 epochs (meaning \texttt{loss.backward()} is called 500 times).
\begin{figure}[htbp]
\begin{subfigure}[h]{0.32\textwidth}
\includegraphics[width=\textwidth]{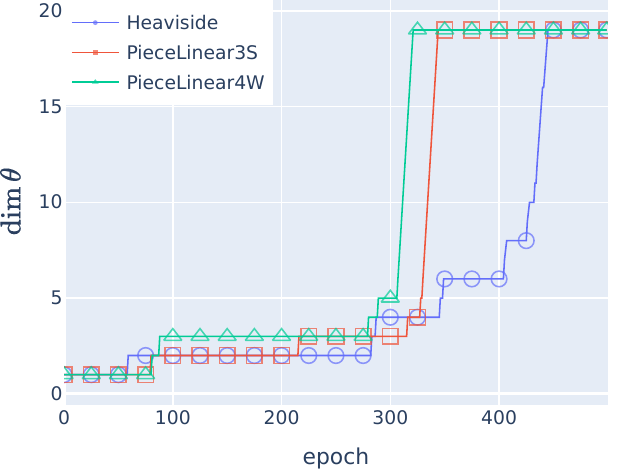}
\caption{Number of Fourier basis}\label{fig:1Dtest-fbasis}
\end{subfigure}\hfill
\begin{subfigure}[h]{0.32\textwidth}
\includegraphics[width=\textwidth]{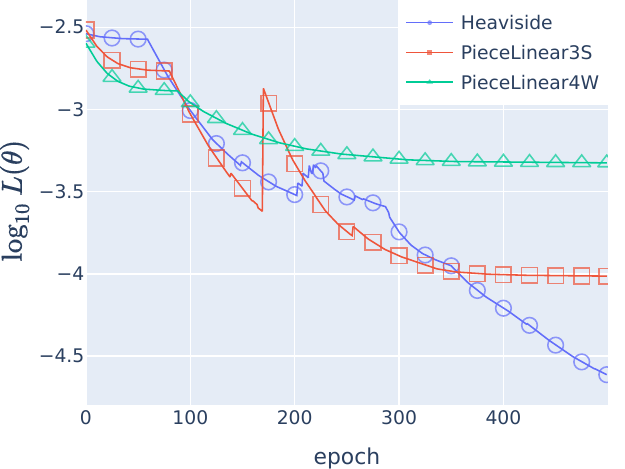}
\caption{Loss}\label{fig:1Dtest-loss}
\end{subfigure}\hfill
\begin{subfigure}[h]{0.32\textwidth}
\includegraphics[width=\textwidth]{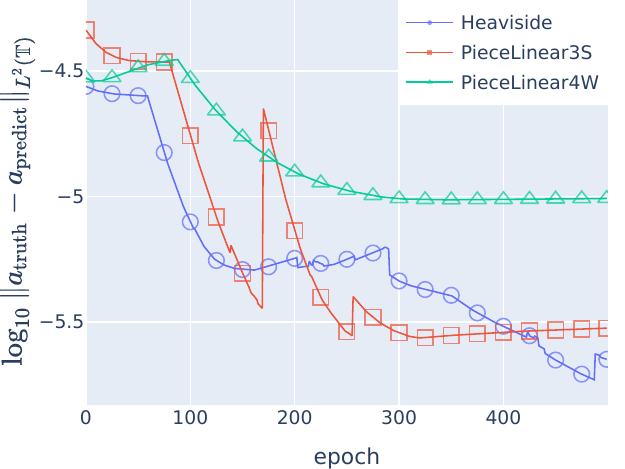}
\caption{Error}\label{fig:1Dtest-error}
\end{subfigure}
\caption{Training logs}
\end{figure}
\Cref{fig:1Dtest-fbasis} shows how the number of Fourier basis functions evolves throughout the training process, and the jumps occur when gradient of the current model parameter \(\nabla_\theta L\) nearly vanishes, which eventually reaches the cap \(\dim\theta\leq19\).
Additionally, \cref{fig:1Dtest-loss} and \cref{fig:1Dtest-error} demonstrate the corresponding reduction in the loss and error, respectively, indicating steady improvement in model performance as training progresses.
The vertical axis of \cref{fig:1Dtest-error} can be interpreted as the logarithm of \cref{eqn:relative-error} since they only differ by a constant.

\subsection{\texorpdfstring{Moving sensors on \(\mathbb{T}^2\) with Discretization \(J=32\)}{Moving sensors on T2 with discretization J=32}}
For the 2D problem, we utilize MNIST grayscale images \citep{lecun1998gradient} as test cases to evaluate the performance of the conductivity reconstruction using moving sensors.
The original MNIST images, which are \(28 \times 28\) pixels with grayscale values in the range \([0, 1]\), are first resized to \(32 \times 32\) (the nearest dyadic number) using \texttt{torchvision.transforms.Resize((32, 32))} \citep{torchvision2016}.
We then apply an affine transformation to scale the pixel values to the range \([\texttt{1e-2}, \texttt{2e-2}]\), representing conductivity values.
This transformation is defined as:
\[ a(x_i, y_j) = \frac{p(i,j)+1}{100}, \]
where \(p(i,j)\) denotes the pixel value at the \(i\)-th row and \(j\)-th column of the grayscale image.

In this transformation:
\begin{itemize}
    \item Pixels in the background of the MNIST images (with values near 0) are mapped to a conductivity of approximately 0.01, corresponding to purple regions in \cref{fig:mnist00-truth}.
    \item Pixels corresponding to the handwritten digits (with values near 1) are mapped to a conductivity of approximately 0.02, corresponding to yellow regions in \cref{fig:mnist00-truth}.
\end{itemize}
This transformation ensures that the conductivity values fall within a reasonable range, preventing overly fast heat diffusion during the simulation, which could result in insufficient data being collected within each sensor's high-sensitivity window.

The moving sensors are then deployed to capture temperature data from the forward PDE solver using conductivity as these transformed MNIST images.
It is important to note that the traditional machine learning concepts of training and testing sets do not apply here, as the solver does not require label knowledge of the MNIST images -- all the information it processes comes solely from the sensor measurements (and initial temperature).

Finally, to evaluate the reconstruction performance, we compare the results obtained from the moving sensor setup with those from an alternative configuration involving 16 and 64 static sensors, respectively.
The comparison in the video \href{https://youtu.be/Z9h16V2FRss}{https://youtu.be/Z9h16V2FRss} allows us to assess the impact of sensor mobility on the accuracy of the conductivity reconstruction.

\subsubsection{4 moving sensors}
\cref{fig:mnist00} illustrates the forward problem using the transformed \texttt{MNIST[0]} image as \(a_\texttt{truth}\).
We obtain the temperature dynamics on a sufficiently dense time grid in \(t\in[0,1]\) and show only the representative snapshots at \(t=0\), \(t=\tfrac{1}{2}\), \(t=1\) respectively in \cref{fig:mnist00-forward}.
\begin{figure}[htbp]
\begin{subfigure}[h]{0.24\textwidth}
\includegraphics[width=\textwidth]{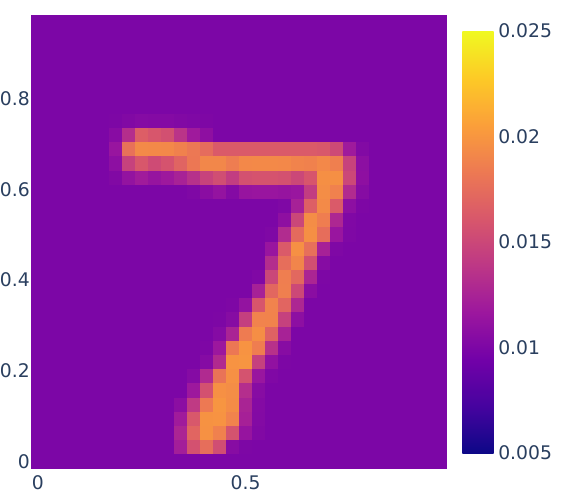}
\caption{\(a_\texttt{truth}\)}
\label{fig:mnist00-truth}
\end{subfigure}\hfill
\begin{subfigure}[h]{0.75\textwidth}
\includegraphics[width=\textwidth]{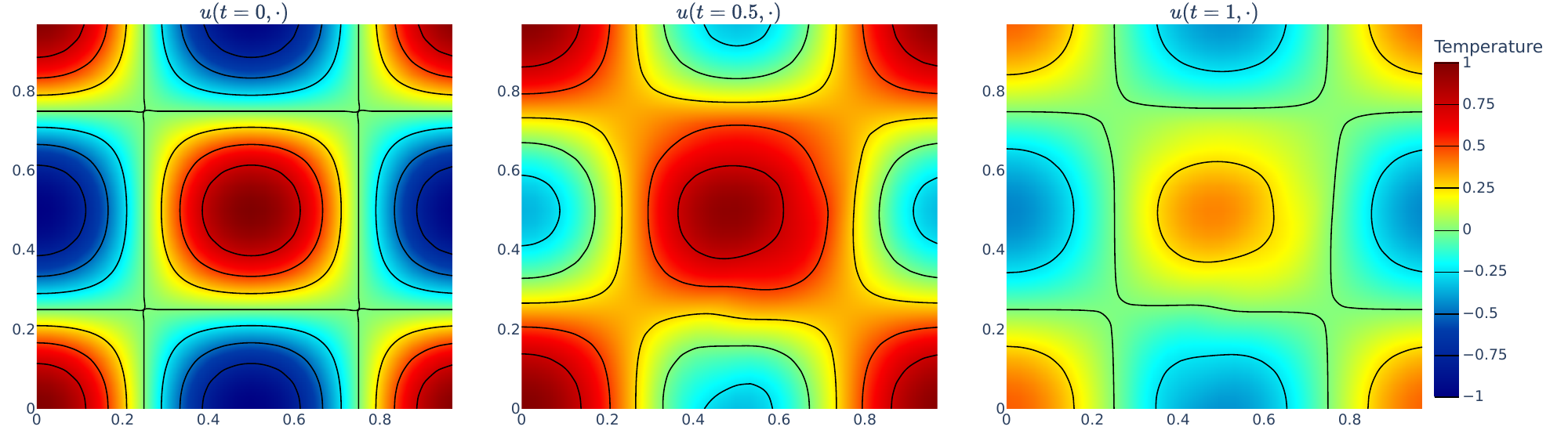}
\caption{Temperature dynamics snapshots at \(t=0\), \(t=\tfrac{1}{2}\), \(t=1\)}
\label{fig:mnist00-forward}
\end{subfigure}
\caption{\texttt{MNIST[0]} as \(a_\texttt{truth}\), \eqref{eqn:heat-init2} as initial condition and \eqref{eqn:heat-source2} as external heat source}
\label{fig:mnist00}
\end{figure}
Next, according to the configuration in \cref{fig:sensor-init_positions}, the 4 moving sensors record temperature in \cref{fig:4m-mnist00-measurement}.
\begin{figure}[htbp]
\begin{subfigure}[h]{\textwidth}
\centering
\includegraphics[width=0.45\textwidth]{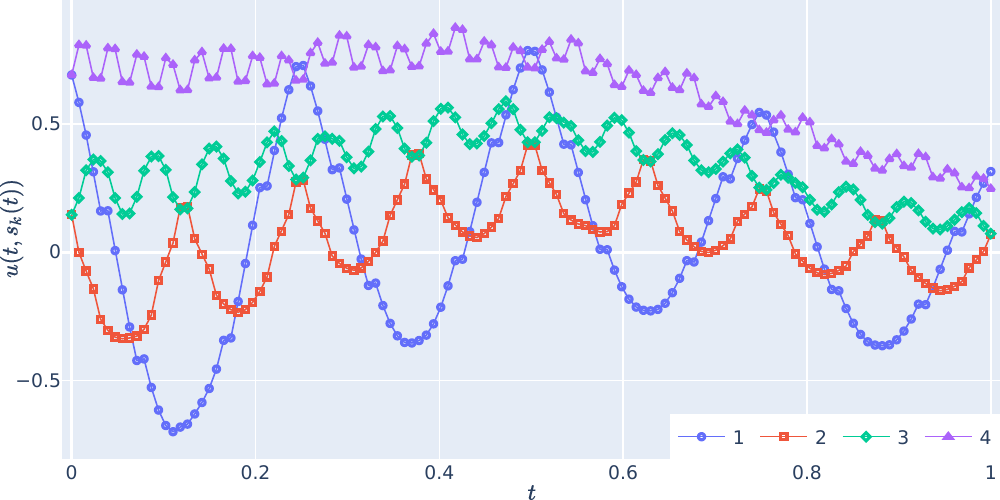}
\caption{Temperature measurements from 4 moving sensors}
\label{fig:4m-mnist00-measurement}
\end{subfigure}
\begin{center}
\begin{subfigure}[h]{0.44\textwidth}
\includegraphics[width=\textwidth]{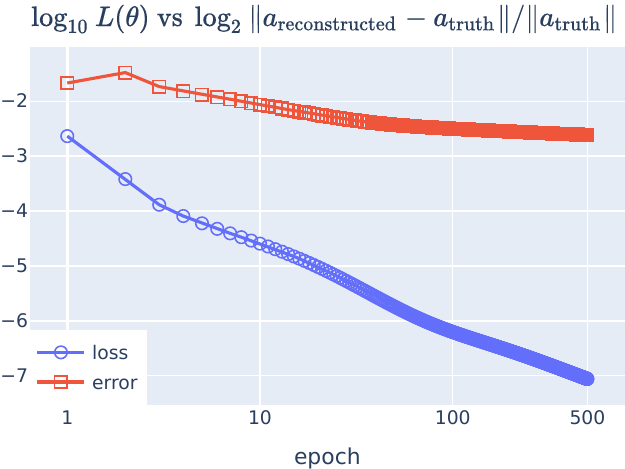}
\caption{Loss and error decay}
\label{fig:4m-mnist00-loss-error}
\end{subfigure}
\begin{subfigure}[h]{0.44\textwidth}
\includegraphics[width=\textwidth]{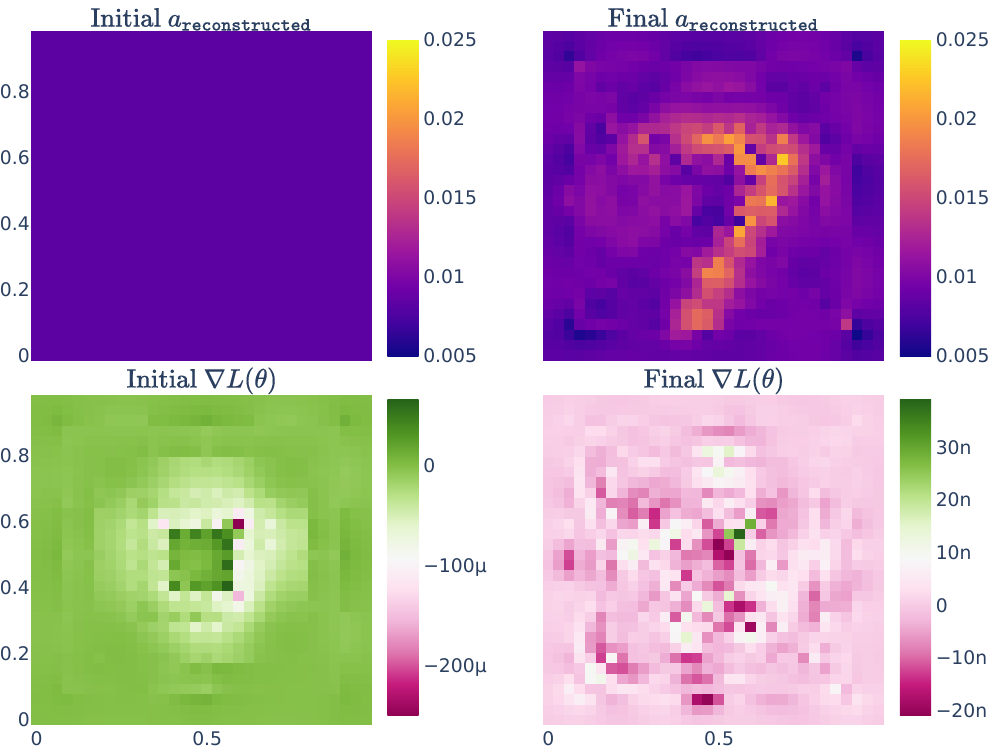}
\caption{Reconstruction and gradient of loss}
\label{fig:4m-recon-grads}
\end{subfigure}
\end{center}
\begin{center}
\begin{subfigure}[h]{0.44\textwidth}
\includegraphics[width=\textwidth]{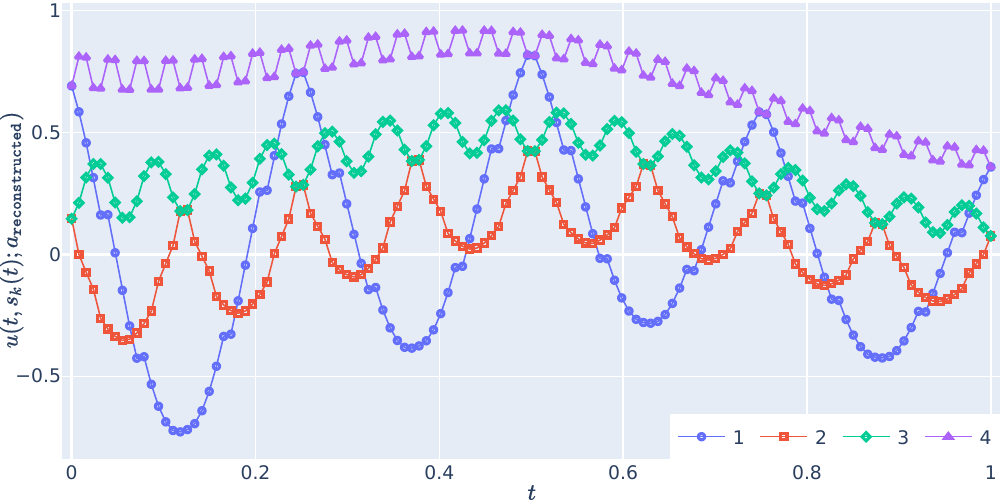}
\caption{Initial simulated measurement}
\label{fig:4m-comp-init}
\end{subfigure}
\begin{subfigure}[h]{0.44\textwidth}
\includegraphics[width=\textwidth]{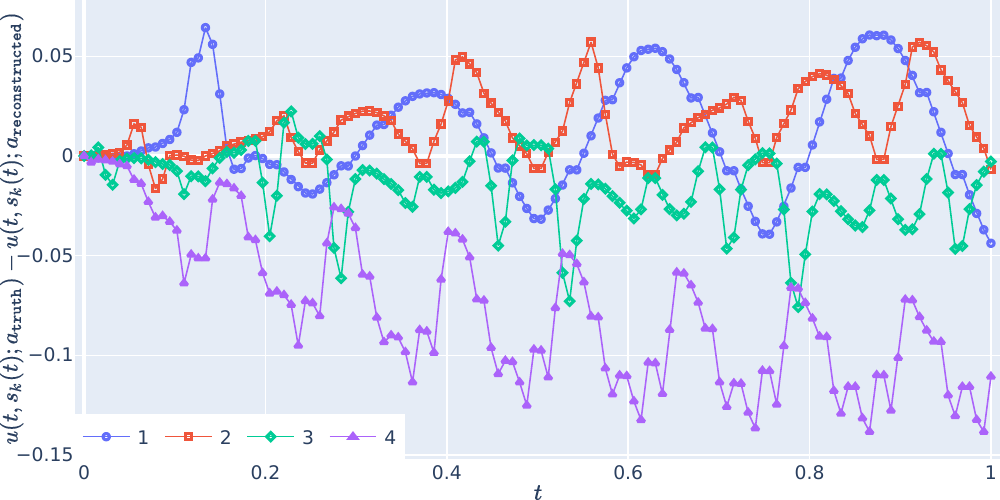}
\caption{Initial residual}
\label{fig:4m-residual-init}
\end{subfigure}
\end{center}
\begin{center}
\begin{subfigure}[h]{0.44\textwidth}
\includegraphics[width=\textwidth]{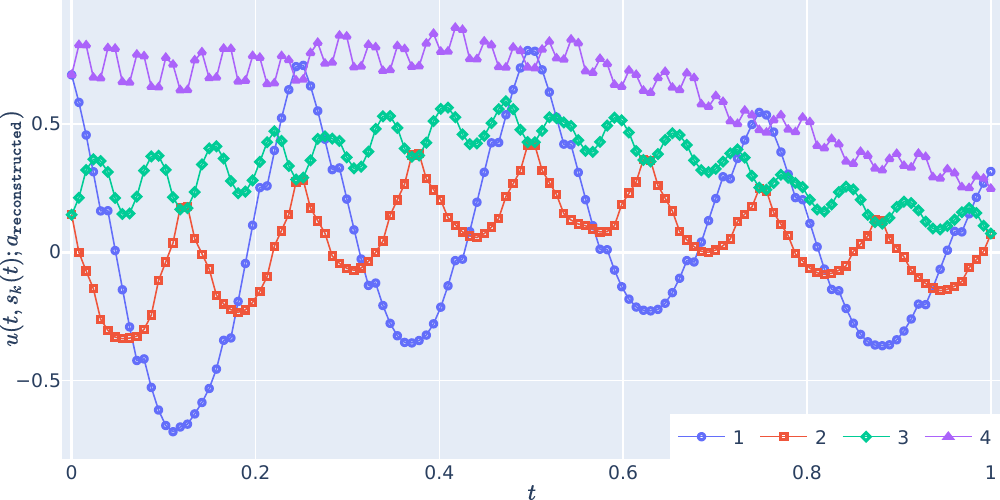}
\caption{Final simulated measurement}
\label{fig:4m-comp-final}
\end{subfigure}
\begin{subfigure}[h]{0.44\textwidth}
\includegraphics[width=\textwidth]{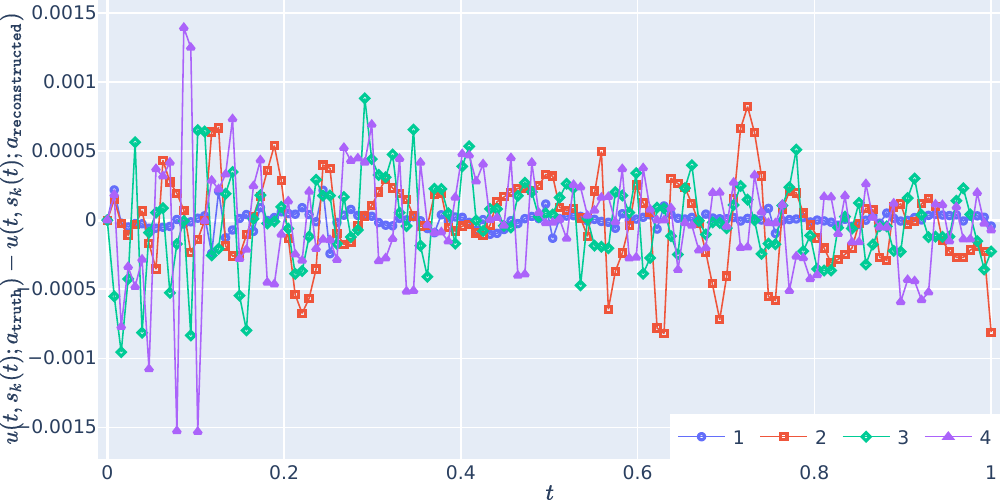}
\caption{Final residual}
\label{fig:4m-residual-final}
\end{subfigure}
\end{center}
\caption{Training log on \texttt{\texttt{MNIST[0]}}}
\end{figure}
Then, we initiate the gradient descent with a constant initial guess (top-left of \cref{fig:4m-recon-grads}), which serves as a conservative lower bound for the true conductivity.
\Cref{fig:4m-comp-init} is the simulated measurement taking \(a_\texttt{reconstructed}\) as the constant initial guess, but there is a significant discrepancy when compared to \cref{fig:4m-mnist00-measurement}, and their differences are displayed in \cref{fig:4m-residual-init}.
The gradient of the loss function, depicted in the bottom-left of \cref{fig:4m-recon-grads}, is used to guide the descent direction.
By following this gradient over 500 epochs, we reach the final estimate, resulting in the reconstructed conductivity \(a_\texttt{reconstructed}\), shown in the top-right of \cref{fig:4m-recon-grads}.
The simulated measurement taking the final estimate as \(a_\texttt{reconstructed}\) is shown in \cref{fig:4m-residual-final}, and the corresponding residual \cref{fig:4m-residual-final} is much smaller, indicating improved accuracy in the reconstruction process.

\subsubsection{Baseline: 16 and 64 equidistant static sensors}
For a fair comparison, the setup of forward problem is exactly the same as the 4 moving sensors, and the only difference is the placement of the sensors, isolating the impact of sensor configuration on the measurement and reconstruction performance.
\begin{figure}[htbp]
\begin{center}
\begin{subfigure}[h]{0.3\textwidth}
\includegraphics[width=\textwidth]{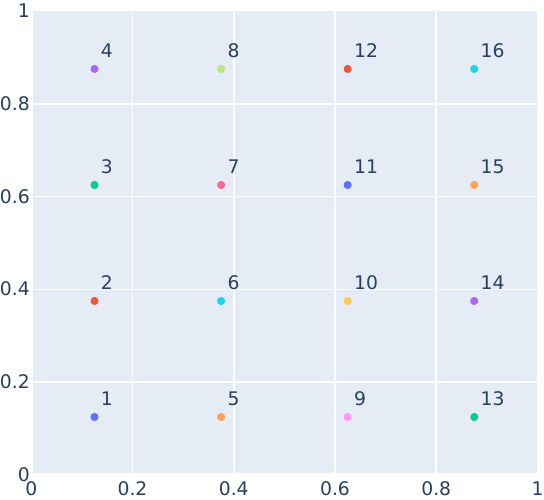}
\caption{16 sensor positions}
\label{fig:16-mnist00-forward}
\end{subfigure}\qquad
\begin{subfigure}[h]{0.5\textwidth}
\includegraphics[width=\textwidth]{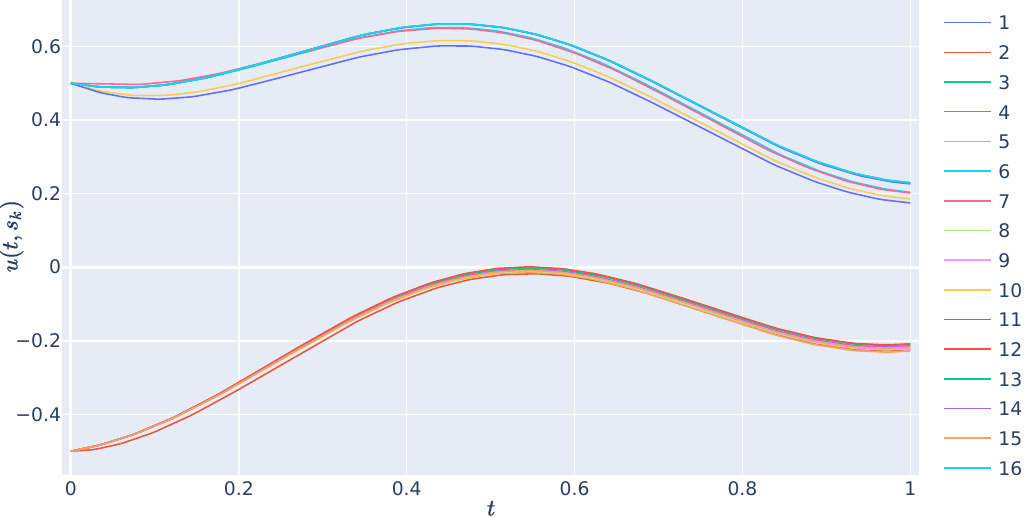}
\caption{Sensor measurements}
\label{fig:16-mnist00-init}
\end{subfigure}
\end{center}
\begin{center}
\begin{subfigure}[h]{0.4\textwidth}
\includegraphics[width=\textwidth]{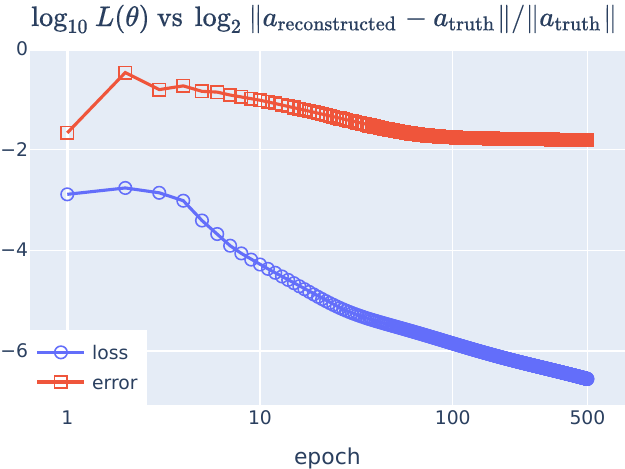}
\caption{Loss and error decay}
\label{fig:16-mnist00-loss-error}
\end{subfigure}\qquad
\begin{subfigure}[h]{0.4\textwidth}
\includegraphics[width=\textwidth]{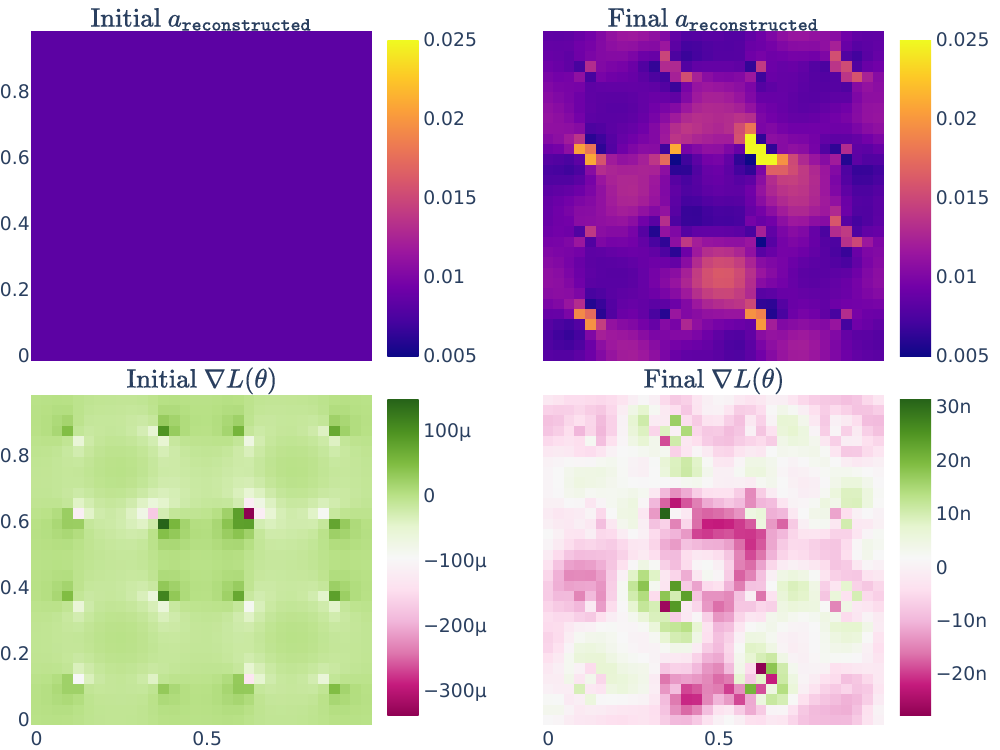}
\caption{Reconstruction and gradient of loss}
\label{fig:16-recon-grads}
\end{subfigure}
\end{center}
\begin{center}
\begin{subfigure}[h]{0.4\textwidth}
\includegraphics[width=\textwidth]{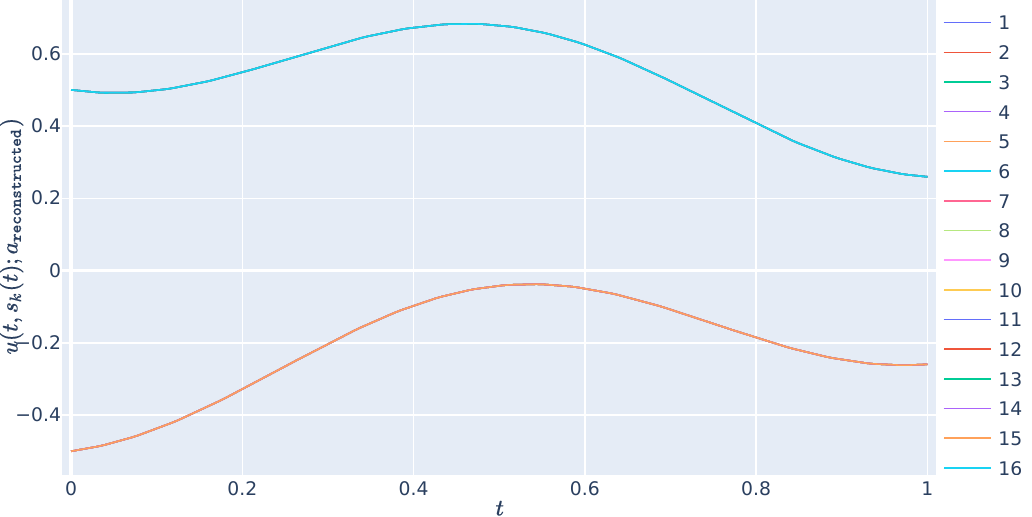}
\caption{Initial simulated measurement}
\label{fig:16-comp-init}
\end{subfigure}\qquad
\begin{subfigure}[h]{0.4\textwidth}
\includegraphics[width=\textwidth]{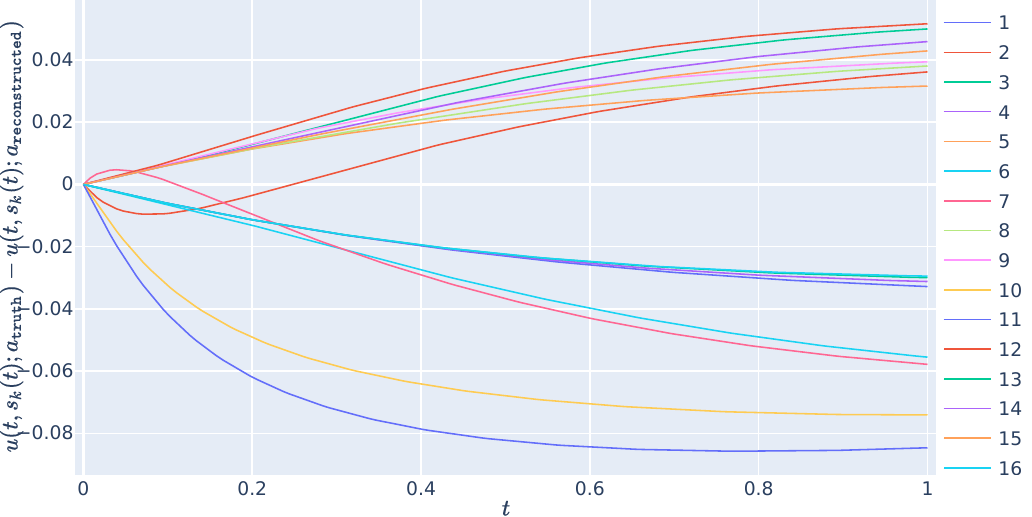}
\caption{Initial residual}
\label{fig:16-residual-init}
\end{subfigure}
\end{center}
\begin{center}
\begin{subfigure}[h]{0.4\textwidth}
\includegraphics[width=\textwidth]{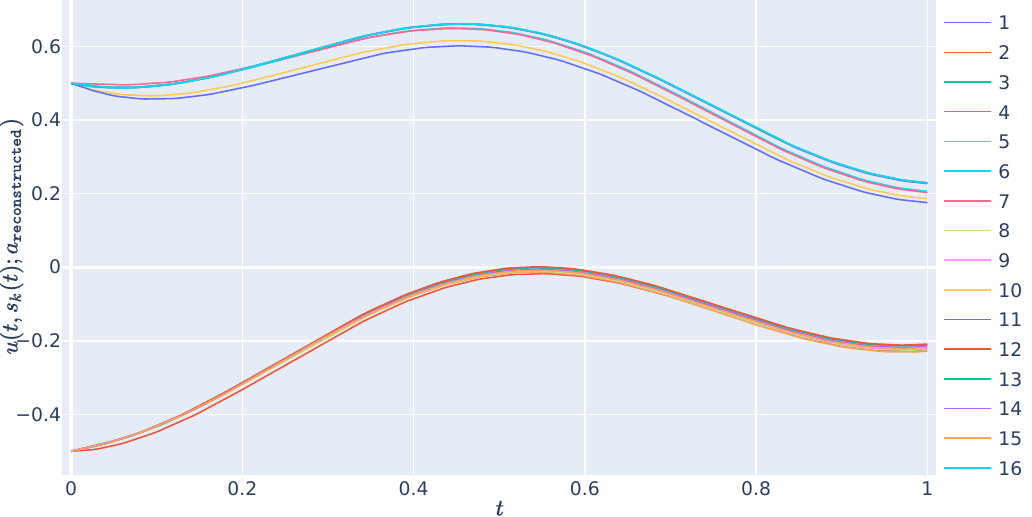}
\caption{Final simulated measurement}
\label{fig:16-comp-final}
\end{subfigure}\qquad
\begin{subfigure}[h]{0.4\textwidth}
\includegraphics[width=\textwidth]{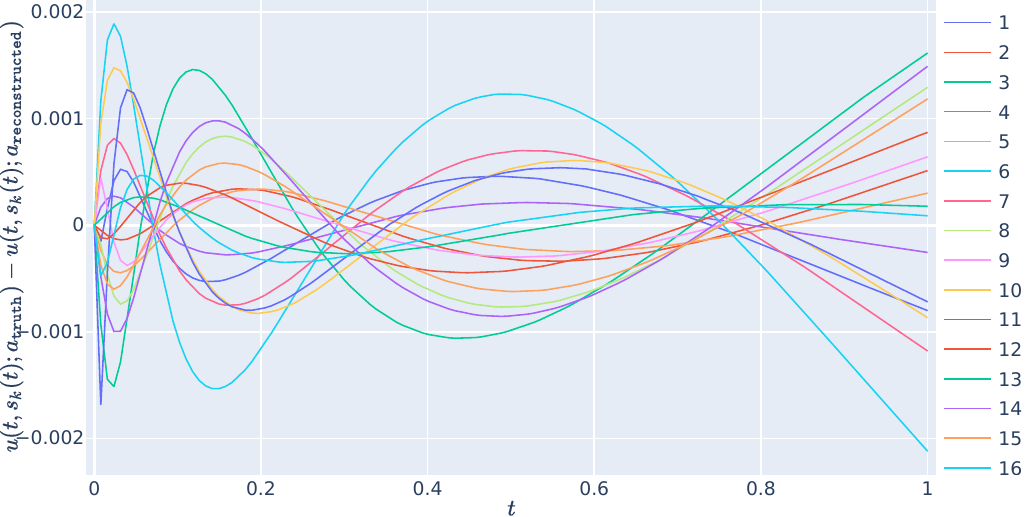}
\caption{Final residual}
\label{fig:16-residual-final}
\end{subfigure}
\end{center}
\caption{16 static sensors}
\label{fig:16-mnist00}
\end{figure}
Notice that only two curves appear in \autoref{fig:16-comp-init} because all measurements overlap to them due to the symmetry in initial condition, sensor positions and the initial guess \(a_\texttt{reconstructed}\) (top-left of \autoref{fig:16-recon-grads}).
\begin{figure}[htbp]
\begin{center}
\begin{subfigure}[h]{0.3\textwidth}
\includegraphics[width=\textwidth]{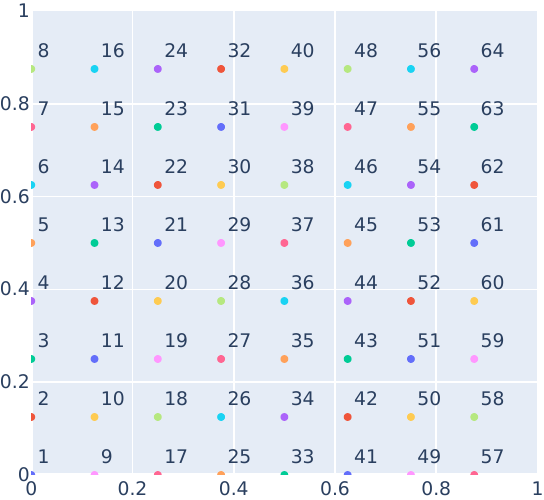}
\caption{64 sensor positions}
\label{fig:64-mnist00-forward}
\end{subfigure}\qquad
\begin{subfigure}[h]{0.5\textwidth}
\includegraphics[width=\textwidth]{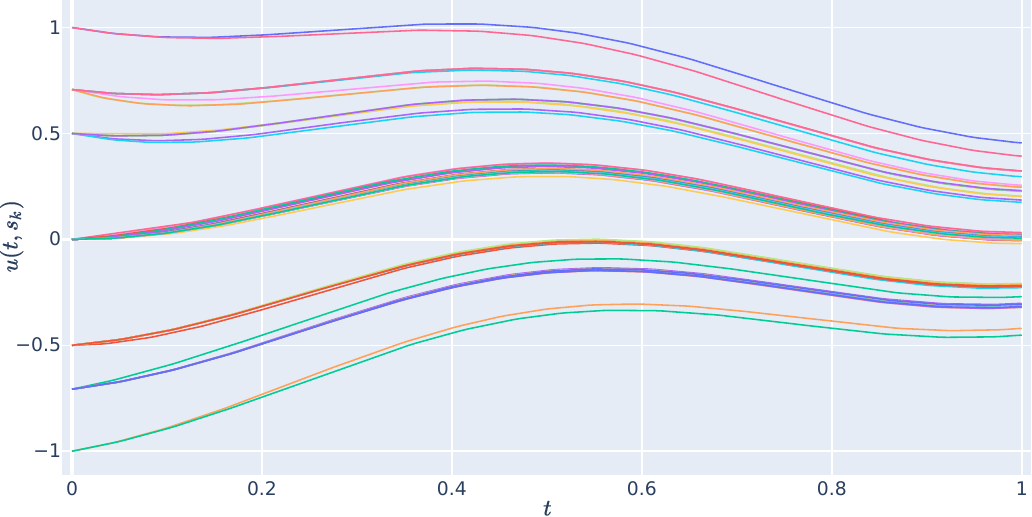}
\caption{Sensor measurements}
\label{fig:64-mnist00-init}
\end{subfigure}
\end{center}
\begin{center}
\begin{subfigure}[h]{0.4\textwidth}
\includegraphics[width=\textwidth]{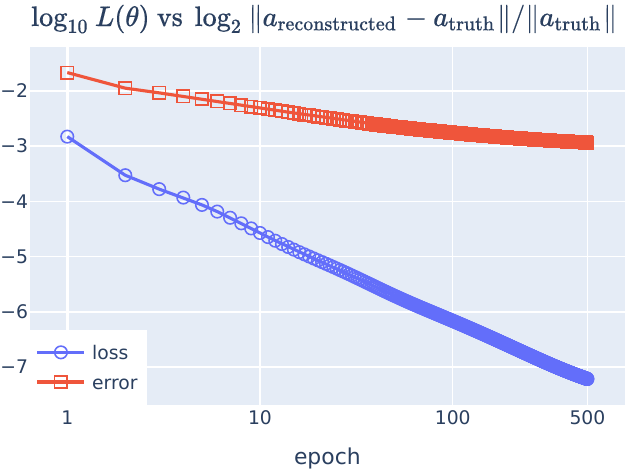}
\caption{Loss and error decay}
\label{fig:64-mnist00-loss-error}
\end{subfigure}\qquad
\begin{subfigure}[h]{0.4\textwidth}
\includegraphics[width=\textwidth]{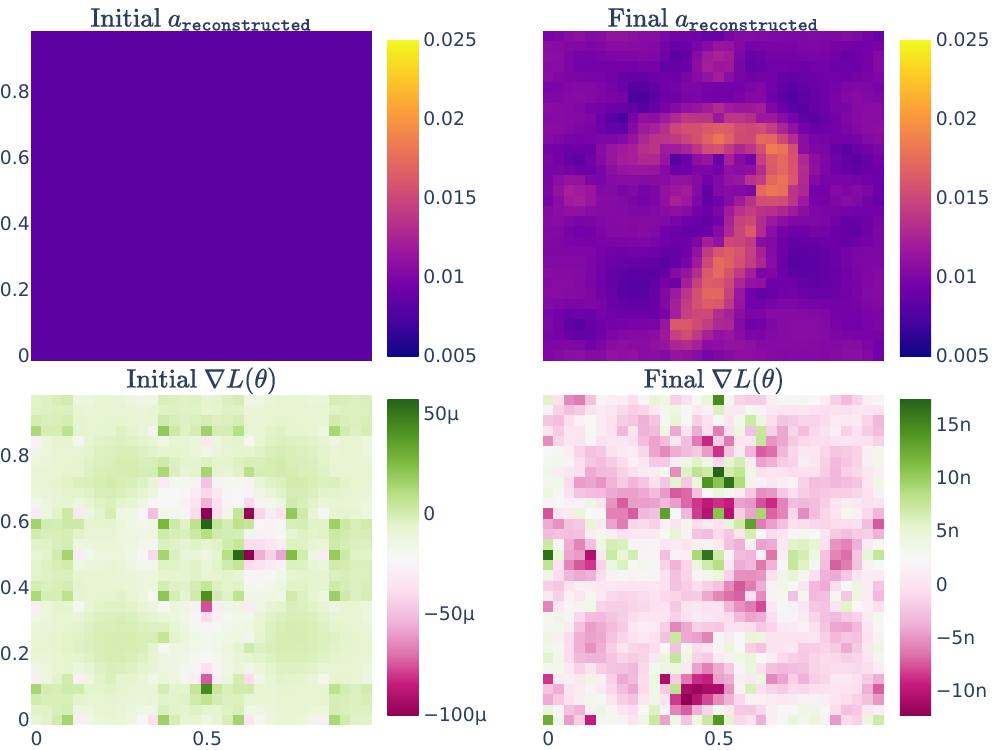}
\caption{Reconstruction and gradient of loss}
\label{fig:64-recon-grads}
\end{subfigure}
\end{center}
\begin{center}
\begin{subfigure}[h]{0.4\textwidth}
\includegraphics[width=\textwidth]{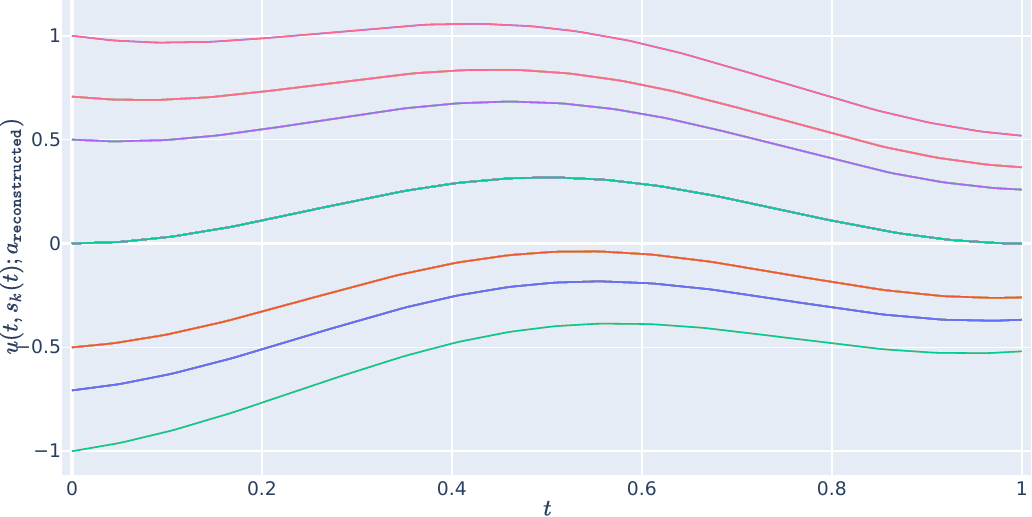}
\caption{Initial simulated measurement}
\label{fig:64-comp-init}
\end{subfigure}\qquad
\begin{subfigure}[h]{0.4\textwidth}
\includegraphics[width=\textwidth]{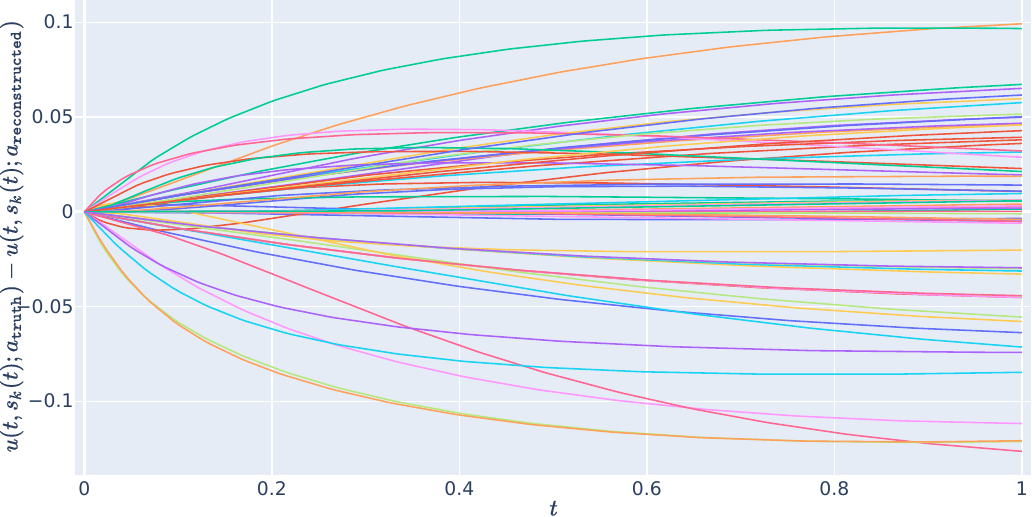}
\caption{Initial residual}
\label{fig:64-residual-init}
\end{subfigure}
\end{center}
\begin{center}
\begin{subfigure}[h]{0.4\textwidth}
\includegraphics[width=\textwidth]{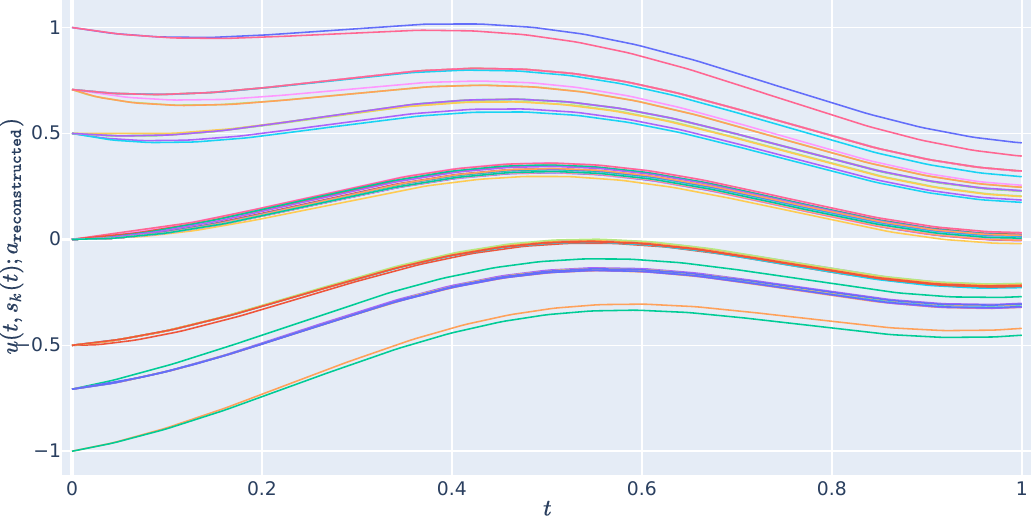}
\caption{Final simulated measurement}
\label{fig:64-comp-final}
\end{subfigure}\qquad
\begin{subfigure}[h]{0.4\textwidth}
\includegraphics[width=\textwidth]{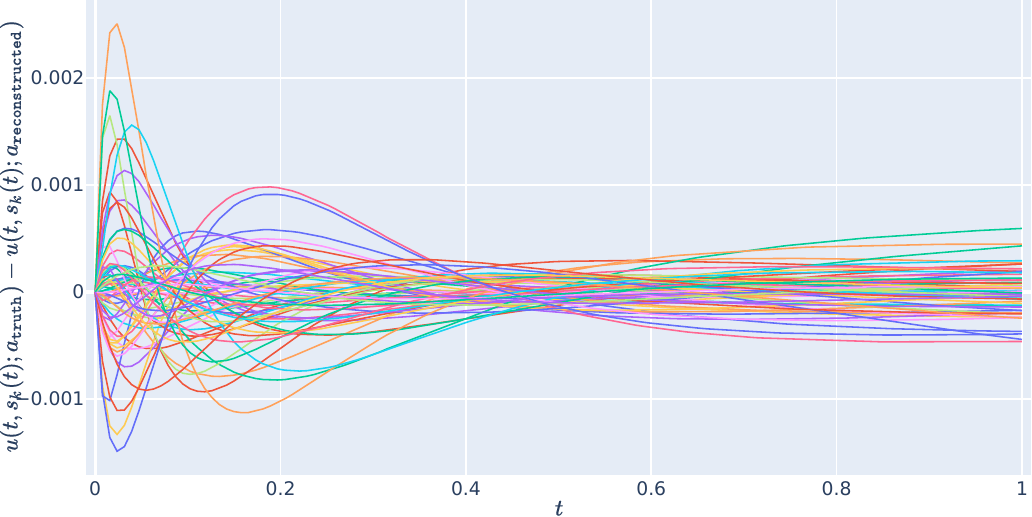}
\caption{Final residual}
\label{fig:64-residual-final}
\end{subfigure}
\end{center}
\caption{64 static sensors}
\label{fig:64-mnist00}
\end{figure}

The reconstruction using 16 static sensors is quite poor (top-right of \autoref{fig:16-recon-grads}).
While increasing the number of static sensors to 64 significantly improves the reconstruction accuracy (top-right of \autoref{fig:64-recon-grads}), the trade-off is the substantially higher number of sensors required.

To compare three different configurations of sensor placements effectively, we focus on evaluating the reconstruction error at a given loss level, rather than tracking the training epoch.
Specifically, we examine how the reconstruction error typically behaves as the loss decreases, as shown in \autoref{fig:2D-loss-error}.
\begin{figure}
\centering
\includegraphics[width=0.9\linewidth]{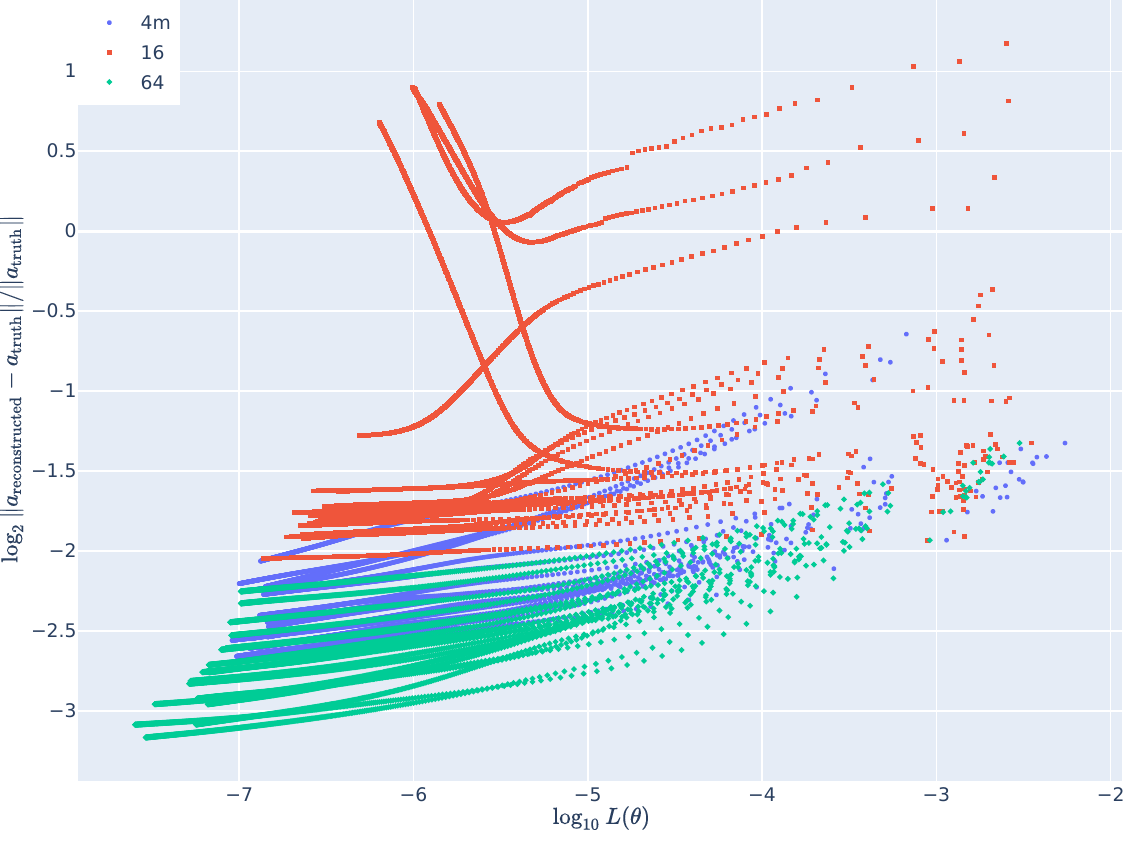}
\caption{Loss error relation based on first 21 images of MNIST}
\label{fig:2D-loss-error}
\end{figure}
The primary interval of interest is for loss values below \(10^{-6}\), as gradient descent generally minimizes the loss to this range (\autoref{fig:4m-mnist00-loss-error}, \autoref{fig:16-mnist00-loss-error} and \autoref{fig:64-mnist00-loss-error}).
Notably, the worst-case performance of the 4 moving sensors still outperforms the best-case scenario of 16 static sensors in terms of relative reconstruction error, despite recording no more amount of data.
Furthermore, since the final residual is at the level of \(10^{-3}\) (\autoref{fig:4m-residual-final}, \autoref{fig:16-residual-final} and \autoref{fig:64-residual-final}), measurements are noise-resilient up to the same scale of additive noise.
The code for \(\mathbb{T}^2\) is available at
\href{https://github.com/yugt/inverse-heat}{https://github.com/yugt/inverse-heat}.
\section{Conclusion}

In this work, we present a novel and efficient method for reconstructing thermal conductivity in both 1D and 2D domains using moving sensors.
Our approach leverages the mobility of sensors to maximize spatial coverage during the most sensitive phase of temperature measurement, allowing for accurate conductivity reconstruction with fewer sensors.
The use of automatic differentiation significantly reduces computational costs, enabling scalable and robust conductivity recovery.
Extensive simulations on both 1D and 2D domains, including evaluations on transformed MNIST images, demonstrate the method’s effectiveness.
Notably, even with only four moving sensors, the reconstruction accuracy outperforms configurations with 16 or even 64 static sensors, indicating the advantage of sensor mobility in sparse data collection scenarios.
The proposed approach, combining reduced computational burden and improved reconstruction quality, offers a practical and versatile solution for inverse problems in heat equations, applicable to a wide range of real-world thermal conductivity mapping tasks.
Future work may extend these techniques to domains with complex boundary conditions or multi-scale conductivity variations, further enhancing the applicability of this method.
\appendix
\section*{Acknowledgements}
The author sincerely thanks the referees for their useful comments.
The author also extends gratitude to Agnid Banerjee for proving  \cref{lemma:lipschitz-dep}, and to Kookjin Lee for providing access to a private computing node on the ASU Sol Supercomputer \citep{HPC:ASU23}, without which this work will take thousands of CPU hours serially.
The author acknowledge the usage of ChatGPT to fix grammar and general style and to perform proof-reading.
% Furthermore, \cref{eqn:forward-Fourier-diff} is fixed by ChatGPT o1-mini, which originally contained a calculation mistake.

\bibliography{references}

\section{Animation}
\href{https://youtu.be/Z9h16V2FRss}{https://youtu.be/Z9h16V2FRss} shows the video of training logs on multiple instances of MNIST image as ground truth conductivity.
All the 4 input instance are tested on the three methods: 4 moving sensors, 16 equidistant static sensors and 64 equidistant static sensors.
The snapshots are also shown in \cref{extra-fig}.

\section{Code Snippets}
The code for the 1D circle is primarily pedagogical, serving as a foundation and inspiration for applying similar methods to 2D torus.
In this implementation, \texttt{self.ode\_func} corresponds to the ODE \begin{equation}\label{eqn:forward-ODE}
\dot{\mathbf{u}}(t) = A(a) \mathbf{u}(t) + \mathbf{f}(t),
\end{equation}
while the external heat source \texttt{self.heat\_source} is defined as in
\begin{equation}\label{eqn:heat-source}
f(t,x) = \sin(\pi t), \qquad (t,x)\in[0,\infty)\times\mathbb{T}.
\end{equation}
\texttt{self.forward\_PDE} first fills the matrix \(A(a)\) based on the input \(a(\cdot)\) on \(\mathbb{T}\), and its output represents \(u(t,x)\) on \([0,1] \times \mathbb{T}\), with each temperature measurement \(u(t_m, x_j)\) having automatic differentiation enabled.

In contrast, the code for the 2D torus is thoroughly tested, well-maintained, and highly recommended as the go-to reference.
It is hosted on
\begin{center}
\href{https://github.com/yugt/inverse-heat}{https://github.com/yugt/inverse-heat}.
\end{center}

\begin{lstlisting}[
    caption={InverseHeat1D}, 
    label={lst:inverse1D}, 
    language=Python,
    basicstyle=\footnotesize
]
import torch
from torchdiffeq import odeint

torch.backends.cudnn.deterministic = True
device = torch.device("cuda:0" if torch.cuda.is_available() else "cpu")

class InverseHeat1D:
    fs2t = lambda self, fm: self.fbasis[:, :fm.shape[0]] @ fm
    
    def __init__(self, J=100, M=None, fmodes=19) -> None:
        if M is None:
            M = J
        self.t = torch.linspace(0, 1, M+1).to(device)
        self.x = torch.linspace(0, 1, J+1)[:-1].to(device)
        self.f = torch.sin(torch.pi * self.t).unsqueeze(-1) @\
                    torch.ones_like(self.x).unsqueeze(0)
        self.u0 = torch.sin(2 * torch.pi * self.x)
        self.fbasis = torch.stack([ 
                torch.cos(_ * 2*torch.pi * self.x) if _ % 2==0
                else torch.sin(_ * 2*torch.pi * self.x) 
            for _ in range(fmodes) ]).T.to(device)
        self.fbasis /=  torch.sqrt(torch.diag(self.fbasis.T @ self.fbasis))
        
    def heat_source(self, t, x):
        return torch.sin(torch.pi * t) * torch.ones_like(x)

    def ode_func(self, t, x, A=None):
        return A @ x + self.heat_source(t, x)  # u'(t) = A(a)u(t) + f(t)

    def forward_PDE(self, a=None):
        if a is None:
            a = self.a
        A = torch.diag(a[:-1], diagonal= 1) +\
            torch.diag(a[:-1], diagonal=-1)
        A[0, -1] = a[-1]
        A[-1, 0] = a[-1]
        A -= torch.diag(a + torch.roll(a, shifts=1))
        # Matrix A(a) is prepared
        return odeint(lambda t, u: self.ode_func(t, u,
                                    A=A * self.x.shape[0]**2),
                      self.u0, self.t)

def formulate(a:torch.Tensor, name:str):
    solver = InverseHeat1D(J=100, fmodes=19)
    solver.a = a
    m = torch.zeros((solver.t.shape[0], solver.x.shape[0]),
                    dtype=torch.bool)
    for _ in range(1, m.shape[0]):
        m[_, _ - 1] = True
    solver.sensor = m
    solver.observe = solver.forward_PDE()[solver.sensor]
    solver.name = name
    return solver

def trainer(solver, epoch=500):
    log = { 'loss':[], 'error':[], 'fbasis':[] }
    gradescent = lambda params: torch.optim.SGD([params], lr=50, 
            momentum=0.1, dampening=0.1, weight_decay=0)
            ## steepest descent
    afm = -41*torch.ones(1).to(device)
    afm.requires_grad = True
    l2 = torch.nn.MSELoss(reduction='mean').to(device)
    opt = gradescent(afm)
    for _ in range(epoch):
        opt.zero_grad()
        ahat = torch.exp(solver.fs2t(afm))
        predict = solver.forward_PDE(ahat)[solver.sensor]
        loss = l2(solver.observe, predict)
        loss.backward()
        opt.step()
        error = ((ahat - solver.a)**2).sum() / solver.a.shape[0]
        print(afm.detach().cpu(),
            afm.grad.detach().cpu(),
            loss.detach(), error.detach())
        log['loss'].append(torch.log10(loss).item())
        log['error'].append(torch.log10(error).item())
        log['fbasis'].append(afm.shape[0])
        if (afm.grad**2).sum() < 1e-7 and
            afm.shape[0] < solver.fbasis.shape[1]:
            temp = torch.zeros(afm.shape[0] + 1).to(device)
            temp[:-1] = afm.detach()
            afm = temp.requires_grad_()
            opt = gradescent(afm)
\end{lstlisting}

\section{Extra figures for 2D}\label{extra-fig}
The following figures mentions ``4m''  in their titles, which is an abbreviation of ``4 moving''.

\foreach \index in {04, 12, 13} {
    \foreach \sensors in {4m, 16, 64} {
        \begin{figure}[htbp]
            \begin{subfigure}[h]{\textwidth}
                \centering
                \includegraphics[width=0.8\textwidth, height=0.16\linewidth]{appendix/\sensors-\index-forward.pdf}
                \caption{Temperature dynamics and measurements from \sensors\ sensors}
            \end{subfigure}
            \begin{subfigure}[h]{\textwidth}
                \centering
                \includegraphics[width=0.8\textwidth]{appendix/\sensors-\index-e00.pdf}
                \caption{Initial guess: constant conductivity}
            \end{subfigure}\hfill
            \begin{subfigure}[h]{\textwidth}
                \centering
                \includegraphics[width=0.8\textwidth]{appendix/\sensors-\index-e5h.pdf}
                \caption{Reconstruction after 500 GD steps}
            \end{subfigure}
            \caption{Traning log for \texttt{MNIST[\index]}}
        \end{figure}
    }
}

\end{document}